\documentclass[a4paper, 11pt]{article}

\usepackage{fullpage}
\usepackage[margin=1in]{geometry}

\usepackage[utf8]{inputenc}
\usepackage[english]{babel} 

\usepackage[dvipsnames]{xcolor}

\usepackage{amsmath}
\usepackage{amsthm}
\usepackage{amssymb}
\usepackage{graphicx}
\usepackage[ruled]{algorithm2e}
\usepackage{mathtools}
\usepackage{dsfont}

\usepackage{bbm}
\usepackage{caption}
\usepackage{subcaption}
\usepackage{multirow}
\usepackage{array}
\usepackage{enumitem}
\usepackage{accents}

\usepackage[pdftex,pdfborder={0 0 0}, 
colorlinks=true, 
linkcolor=blue, 
citecolor=blue,
urlcolor=blue,
pagebackref=false, 
]{hyperref}

\newcommand{\E}{\mathbb{E}}
\renewcommand{\P}{\mathbb{P}}
\newcommand{\R}{\mathbb{R}}

\newcommand{\N}{\mathbb{N}}

\newcommand{\Lbb}{\mathbb{L}}

\newcommand{\Xcal}{\mathcal{X}}
\newcommand{\Ycal}{\mathcal{Y}}
\newcommand{\Acal}{\mathcal{A}}

\newcommand{\Mcal}{\mathcal{M}}

\newcommand{\Pcal}{\mathcal{P}}
\newcommand{\Ccal}{\mathcal{C}}

\newcommand{\Ucal}{\mathcal{U}}

\newcommand{\Zcal}{\mathcal{Z}}
\newcommand{\Jcal}{\mathcal{J}}

\newcommand{\efrak}{\mathfrak{e}}

\newcommand{\supp}{\textnormal{supp}}

\newcommand{\Bpi}{\boldsymbol{\pi}}

\renewcommand{\epsilon}{\varepsilon}

\newcommand{\INTDom}[3]{\int_{#2} #1 \,\mathrm{d} #3}

\theoremstyle{plain}
\newtheorem{theorem}{Theorem}[section]
\newtheorem{proposition}[theorem]{Proposition}
\newtheorem{lemma}[theorem]{Lemma}
\newtheorem{corollary}[theorem]{Corollary}

\newtheorem{taggedhypx}{Assumptions}
\newenvironment{taggedhyp}[1]
{\renewcommand\thetaggedhypx{#1}\taggedhypx}
{\endtaggedhypx}

\theoremstyle{definition}
\newtheorem{definition}[theorem]{Definition}

\theoremstyle{remark}
\newtheorem{remark}[theorem]{Remark}

\newcommand{\dd}{\,\mathrm{d}}

\newcommand{\licorm}{licorm}

\begin{document}
\numberwithin{equation}{section}

\title{A Characterization of Law-Invariant and Coherent Risk Measures through Optimal 
Transport}

\author{Riccardo Bonalli\thanks{Université Paris-Saclay, CNRS, CentraleSupélec, Laboratoire des signaux et systèmes, 91190, Gif-sur-Yvette, France.} {}\thanks{Fédération de Mathématiques de CentraleSupélec, 91190, Gif-sur-Yvette, France.} \quad Benoît Bonnet-Weill\footnotemark[1]{ }\footnotemark[2] \quad Laurent Pfeiffer\thanks{Université Paris-Saclay, CNRS, CentraleSupélec, Inria, Laboratoire des signaux et systèmes, 91190, Gif-sur-Yvette, France.} \footnotemark[2]}

\date{\today}

\maketitle

\begin{abstract}
In this article, we propose a novel characterization of law-invariant and coherent risk measures, based on a generalized optimal transportation problem in which the second marginal of the admissible plans is not fixed, but required to lie within a target set of probability measures.
One of the main contributions of this work is a general representation formula for such risk measures, which is closely related to Kusuoka's theorem. When the aforementioned target set is convex, our representation result allows for the systematic derivation of general duality formulas. To illustrate our findings, we explicitly compute the target sets associated with several classical law-invariant coherent risk measures, including the prototypical conditional value at risk and higher moment measures.
\end{abstract}

\setcounter{tocdepth}{1}

\tableofcontents

\section{Introduction}

This article focuses on a new class of risk measures defined through an optimal transport problem. Given $p \in [1,+\infty]$, we denote by $q \in [1,\infty]$ its conjugate exponent, and fix a subset $R \subset\Pcal_q(\R)$ of probability measures with finite moment of order $q$. The risk measure of interest $\rho_R : \Lbb^p(\Omega,\R) \to \R$ is then defined as
\begin{equation*}
\rho_R(X) \triangleq \chi_R(\mathbb{P}_X),
\end{equation*}
where $\mathbb{P}_X \in\Pcal_p(\R)$ stands for the probability distribution of a random variable $X \in\Lbb^p(\Omega,\R)$, and $\chi_R \colon \Pcal_p(\R) \to \R$ is the value function of the following generalized optimal transport problem
\begin{equation} \label{eq:OT_intro}
\chi_R(m)\triangleq
\sup_{\pi \in \Pi(m,R)} \, 
\int_{\R^2} xy \dd \pi(x,y).
\end{equation}
Therein, the set $\Pi(m,R)$ stands for the collection of all transport plans whose first marginal is equal to $m \in \Pcal_p(\R)$, and whose second marginal belongs to $R \subset\Pcal_q(\R)$. The peculiarity of the above problem, in contrast with standard optimal transport investigated deeply e.g.\@ in \cite{AGS,OTAM,villani2009optimal}, lies in the fact that the second marginal of the transport plans appearing in \eqref{eq:OT_intro} is not fixed, but merely required to lie within some prescribed set.

\begin{remark}[An enlightening example]
\label{rmk:Intro}
The prototypical example of coherent risk measures is the so-called \textnormal{Conditional Value at Risk} with a given probability level $\beta \in [0,1)$, defined by 
\begin{equation} \label{eq:dualityCV@R}
\mathrm{CV@R}_{\beta}(X) \triangleq \inf_{t \in \R} \bigg\{ t + \frac{1}{1-\beta} \mathbb{E} \big[ (X-t)_+ \big] \bigg\}.
\end{equation}
We shall below see that the latter fits into our framework and corresponds to the simple case in which $R \triangleq \{ r_\beta \}$, with
\begin{equation*}
r_\beta \triangleq \beta \delta_0
+ (1-\beta) \delta_{1/(1-\beta)}.
\end{equation*}
This fact will be rigorously demonstrated with the help of a duality formula we shall discuss later on, although one may easily convey the underlying intuition when the probability distribution $m \in\Pcal_p(\R)$ of the random variable $X \in\Lbb^p(\Omega,\R)$ is nonatomic. In that case, there exists a real number $t_{\beta} \in \R$ -- called the Value at Risk of $X$ with probability level $\beta$ -- such that 
\begin{equation*}
\mathbb{P}(X < t_{\beta}) = \beta \qquad \text{and} \qquad \mathbb{P}(X\geq t_{\beta})= 1- \beta. 
\end{equation*}
Then, it is commonly known that $\mathrm{CV@R}_{\beta}(X) = \E\big[X \,|\, X \geq t_{\beta} \big]$, see for instance \cite{rockafellar2002conditional}, and it can be shown that the plan defined by
\begin{equation*}
\bar{\pi} \triangleq
\left( m_{\llcorner(-\infty,t_{\beta})} \times \delta_0 \right) +
\left( m_{\llcorner[t_{\beta},+\infty)} \times \delta_{1/(1-\beta)} \right)
\end{equation*}
is optimal for \eqref{eq:OT_intro}, so that 
\begin{equation*}
\begin{aligned}
\chi_{r_{\beta}}(m) & = \int_{\R^2} xy \dd \bar{\pi}(x,y)
\\
& = \frac{1}{1-\beta} \int_{\R} x \dd m_{\llcorner[t_{\beta},+\infty)}(x) \\[0.7em]
& = \frac{\mathbb{E}\big[ X \mathds{1}_{\{X \geq t_{\beta}\}} \big]}{\mathbb{P}\big[ X \geq t_{\beta} \big]} \\
& = \mathbb{E}\big[ X \,|\, X \geq t_{\beta} \big]= \mathrm{CV@R}_{\beta}(X),
\end{aligned}
\end{equation*}
as announced. 
\end{remark} 

\paragraph*{Overview of contributions} The first main contribution of this article is a complete characterization of \textit{law-invariant coherent risk measures} (coined ``licorms'' in the sequel) through the generalized  optimal transport problem \eqref{eq:OT_intro}.
The definitions of law-invariance and coherence for risk measures, which were introduced in \cite{artzner1999coherent}, are recalled in Definition \ref{def:coherence} below. More specifically, we first show in Theorem \ref{thm:kusuoka_converse} that if $R \subset \Pcal_q(\R)$ only contains measures supported in $\R_+$ with expectation equal to 1, then $\rho_R : \Lbb^p(\Omega,\R) \to \R$ is indeed a law-invariant coherent risk measure. These facts will be established via direct proofs, using only basic tools from optimal transport theory, and primarily the gluing lemma (see e.g.\@ \cite[Lemma 5.3.2]{AGS}). 

In Theorem \ref{thm:kusuoka}, we establish a sharp converse of Theorem \ref{thm:kusuoka_converse} stating that if the probability space $(\Omega,\mathcal{A},\P)$ is nonatomic, every law-invariant coherent risk measure is then of the form $\rho_R \colon \Lbb^p(\Omega,\R) \to \R$ with $R \subset\Pcal_q(\R)$ containing only measures supported in $\R_+$ whose expectation is equal to 1. This will be directly deduced from the standard dual representation of coherent risk measures, which states that for every such $\rho : \Lbb^p(\Omega,\R) \to \R$, there exists a convex set of nonnegative random variables $\mathfrak{Y} \subset\Lbb^q(\Omega,\R)$ with expectation equal to 1, such that
\begin{equation} \label{eq:dual_rep}
\rho(X)= \sup_{Y \in \mathfrak{Y}} \mathbb{E} [ XY ],
\end{equation}
for every random variable $X \in\Lbb^p(\Omega,\R)$. We will then show that if $\rho$ also happens to be law-invariant, then it is of the form $\rho_R$ with $R \triangleq \{ \mathbb{P}_Y \, | \, Y \in \mathfrak{Y} \}$. In a very informal fashion, one may think of the optimal transport problem \eqref{eq:OT_intro} as being the deterministic counterpart of \eqref{eq:dual_rep}, and then of Theorem \ref{thm:kusuoka} as a variant of the famed Kusuoka theorem, derived in the seminal paper \cite{kusuoka2001law}. The latter essentially states that any law-invariant coherent risk measure can be represented as the supremum of a family of risk measures, which can be all expressed as (possibly continuous) convex combinations of $(\mathrm{CV@R}_{\beta})_{\beta\in(0,1]}$. For a detailed introduction to such objects, we refer the reader to \cite{Shapiro2021} and \cite[Section 4.5]{follmer2011stochastic}. While the latter and our optimal-transport representation are seemingly different, they happen to be strongly connected as more amply detailed in Remark \ref{remark:kusuoka_comp}.

The second main contribution of this article takes the form of two general duality formulas for the optimal transport problem \eqref{eq:OT_intro}, under the assumption that $R \subset\Pcal_q(\R)$ is convex. Both duality results are stated and proven in Theorem \ref{theo:duality}. It is worth noting that when $R$ is a singleton, such formulas boil down to the well-known Kantorovich duality theorem. As already mentioned, they allow, among other things, to justify rigorously that $\rho_{R}$ coincides with $\mathrm{CV@R}_{\beta}$ as defined in \eqref{eq:dualityCV@R}, when $R= \{ r_{\beta} \}$. The existence of solutions to the dual problem of \eqref{eq:OT_intro} are investigated in Theorem \ref{theo:existence_finitecase} and Theorem \ref{theo:existence_bounded}, under suitable assumptions on the set $R \subset\Pcal_q(\R)$. To the best of our knowledge, results of this kind were unavailable in the optimal transport literature, with the exception of the very recent preprint \cite{nenna2025convergence} (see in particular Proposition 2.2 therein), which focuses solely on discrete measures.

\paragraph*{Related works} Numerous articles have investigated optimal transport-based risk measures in the recent literature, in particular in the framework of Wasserstein Distributionally Robust Optimization (see e.g. \cite{kuhn2019wasserstein,zhao2018data}), a topic which has received a lot of attention lately. We refer the reader to \cite{azizian2023regularization} for a study of regularization techniques in this context. Another example appears in risk quantization by magnitude and propensity \cite{faugeras2024risk}. 
Independently, several works have investigated generalizations and refinements of Kusuoka's representation theorem. Among others, we mention \cite{dentcheva2010kusuoka}, which establishes a Kusuoka representation of high-order dual risk measures, as well as  \cite{shapiro2013kusuoka} that provides a characterization of risk measures taking the form of a convex combination of Conditional Values at Risk, and studies their fundamental properties. We also point to the article \cite{jouini2006law} which proves an extension of Kusuoka's theorem to the case of law-invariant convex risk measures, which need not be coherent a priori. Lastly, a description of those risk measures admitting a Kusuoka representation in general nonatomic probability spaces was also established in \cite{noyan2015kusuoka}.

\paragraph*{Organization of the paper} In Section \ref{section:preliminaries}, we introduce the main notations and preliminary results needed in the sequel. Section \ref{section:ot-rm} investigates risk measures defined through a generalized optimal transport problem, while Section \ref{section:characterization} deals with the optimal transport characterization of law-invariant and coherent risk measures. Duality formulas for optimal transport-based risk measures are then established in Section \ref{section:duality}. Finally, some examples involving the Conditional Value at Risk, higher moment measures and perspectives towards $\phi$-divergences are discussed in Section \ref{section:examples}.


\section{Preliminaries and definition of licorms}
\label{section:preliminaries}


\paragraph{Measures and function spaces}

Given a nonempty closed subset $\Xcal \subset\R$, denote by $\Pcal(\Xcal)$ the set of Borel probability measures on $\Xcal$. Given $m \in \Pcal(\Xcal)$ and $p \in [1,+\infty]$, we let 
\begin{equation*}
\Mcal_p(m) \triangleq
\begin{cases}
\begin{array}{cl}
{\displaystyle \bigg( \int_{\R} |x|^p \dd m(x) \bigg)^{1/p}}
& \text{ if $p < +\infty$} \\[1em]
{\displaystyle \sup_{x \in \text{supp}(m)} |x|} & \text{ if $p= +\infty$}.
\end{array}
\end{cases}
\end{equation*}
Accordingly, we define the subset $\Pcal_p(\Xcal) \subset\Pcal(\Xcal)$ by 
\begin{equation*}
\Pcal_p(\Xcal) \triangleq \Big\{ m \in \Pcal(\Xcal) ~\,\textnormal{s.t.}~ \Mcal_p(m) < +\infty \Big\}, 
\end{equation*}
where for $p=+\infty$, the space $\Pcal_p(\Xcal)$ is simply that of compactly supported measures in $\Xcal$, that we denote by $\Pcal_c(\Xcal)$. We shall likewise say that a subset $R \subset\Pcal_p(\Xcal)$ is \emph{bounded} provided that 
\begin{equation*}
\sup_{r \in R} \, \Mcal_p(r) < +\infty.
\end{equation*}
Note that $\Mcal_p(m) \leq \Mcal_q(m)$ whenever $p \leq q$ by H\"older's inequality, so that $\Pcal_1(\Xcal) \subset \Pcal_p(\Xcal) \subset \Pcal_q(\Xcal)$. With a slight abuse of notation, we define the \emph{expectation} of a measure $m \in\Pcal_p(\Xcal)$ as
\begin{equation*}
\E[m]
\triangleq \int_{\R} x \dd m(x).
\end{equation*}
Let us recall now the definition of the $p$-Wasserstein distance, which is given by  
\begin{equation}
\label{eq:def_wasserstein}
W_p(m_1,m_2) := \inf_{\pi \in \Pi(m_1,m_2)} \,
\bigg( \INTDom{|x_1 - x_2|^2}{\R^2}{\pi(x_1,x_2)} \bigg)^{1/2},
\end{equation}
for every $m_1,m_2 \in \mathcal{P}_p(\R^2)$. Note that this definition makes sense for every $p \in [1,\infty]$ (see, e.g., \cite[Section 3.2]{OTAM} for the case $p=+\infty$). Lastly, we denote by $F_m^{-1} \colon [0,1] \rightarrow [ -\infty,+\infty]$ the right-inverse of the cumulative distribution of a probability measure $m \in \Pcal(\R)$, defined by
\begin{equation*}
F_m^{-1}(t)\triangleq \inf \Big\{ x \in \R ~\,\text{s.t.}~ m\big( (-\infty,x]\big) \geq t \Big\},
\end{equation*}
for every $t \in [0,1]$.

We denote by $\Ccal_p^0(\Xcal)$ the space of continuous functions with $p$-growth, namely the space of all those functions $f \in\Ccal(\Xcal)$ for which
\begin{equation*}
\| f \|_p \triangleq \sup_{x \in \Xcal} \
\frac{|f(x)|}{1+ |x|^p} < +\infty,
\end{equation*}
if $p \in [1,+\infty)$. In the case in which $p= +\infty$, the space $\Ccal_p^0(\Xcal)$ is simply that of bounded and continuous real-valued functions on $\Xcal$, which we also denote by $\Ccal_b(\Xcal)$. Note that the normed spaces $(\Ccal_p^0(\Xcal),\| \cdot \|_p)$ are complete for all $p\in[1,+\infty]$. When $p =+\infty$ this is a well-known fact, whereas for $p \in [1,+\infty)$, it stems from the observation that the linear map $ f \in \Ccal_p^0(\Xcal) \mapsto f/ (1+ |\cdot |^p)$ is a continuous bijection between $\Ccal_p^0(\Xcal)$ and the Banach space $\Ccal_b(\Xcal)$, which implies that $\Ccal_p^0(\Xcal)$ must also be a Banach space.

Finally given $m \in \Pcal_p(\Xcal)$, we have for any $f \in \Ccal_p^0(\Xcal)$ that
\begin{equation*}
\bigg| \INTDom{f(x)}{\Xcal}{m(x)} \bigg|
\leq
\left\{
\begin{aligned}
& \| f \|_p  \, \Big( 1+ \Mcal_p^p(m) \Big) && 
\text{ if } p \in [1,+\infty), \\
& \| f \|_p \; \Mcal_{\infty}(m) & & \text{ if }p=
+\infty.
\end{aligned}
\right.
\end{equation*}
This allows us to see any probability measure in $\Pcal_p(\Xcal)$ as a bounded linear form on $\Ccal_p^0(\Xcal)$, that is, as an element of $\Ccal_p^0(\Xcal)^*$ whose action is given via the duality pairing 
\begin{equation*}
\langle m , f \rangle \triangleq \INTDom{f(x)}{\Xcal}{m(x)}
\end{equation*}
for each $(f,m) \in \Ccal_p^0(\Xcal) \times \Ccal_p^0(\Xcal)^*$.


\paragraph{Random variables and risk measures}

From now on, we fix a probability space $(\Omega, \mathcal{A}, \mathbb{P})$. Given $p \in [1,+\infty)$ and a closed subset $\mathcal{Z}$ of $\R^m$, we let $\Lbb^p(\Omega,\mathcal{Z})$ denote the space of random variables on $(\Omega, \mathcal{A}, \mathbb{P})$ valued in $\mathcal{Z}$ and such that $\mathbb{E}[ | X |^p] <+\infty$, equipped with the norm
\begin{equation*}
\| X \|_p \triangleq \mathbb{E} \big[ |X|^p \big]^{1/p}.
\end{equation*}
When $p=+\infty$, recall that $\Lbb^p(\Omega,\mathcal{Z})$ stands for the space of essentially bounded random variables equipped with the supremum norm. Given a random variable $X \in \Lbb^p(\Omega,\Zcal)$, we shall write $\mathbb{P}_X \in \Pcal_p(\Zcal)$ to refer to its probability distribution. Finally, given $X \in \mathbb{L}^p(\Omega,\R)$, we call right-inverse of the cumulative function of $X$ the function $F_X^{-1} \colon \R \rightarrow [-\infty,+\infty]$ defined as $F_X^{-1}(t) := F_{\P_X}^{-1}(t)$ for every $t\in[0,1]$.

\begin{definition}[Coherent and law-invariant risk measures]
\label{def:coherence}
We call \emph{risk measure} any real-valued mapping $\rho \colon \Lbb^p(\Omega,\R) \rightarrow \R$. A risk measure is said to be \emph{coherent} if it satisfies the following properties for all $X,X' \in \Lbb(\Omega,\R)$, every $\alpha \in \R$, each $\delta \geq 0$, and any $\theta \in [0,1]$.
\begin{enumerate}
\item[(i)] $\rho(X+ \alpha)= \rho(X) + \alpha$ \hfill (Translation invariance)
\item[(ii)] $\rho(\delta X)= \delta \rho(X)$ \hfill (Homogeneity)
\item[(iii)]  $\rho(X_1) \leq \rho(X_2)$ whenever $X_1 \leq X_2$ almost surely \hfill (Monotonicity)
\item[(iv)] $\rho( (1-\theta) X_1 + \theta X_2) \leq (1-\theta) \rho(X_1) + \theta \rho(X_2)$. \hfill (Convexity)
\end{enumerate}
We further say that $\rho$ is \emph{law-invariant} if for all $X_1,X_2 \in \Lbb^p(\Omega,\R)$, it holds that 
\begin{equation*}
\rho(X_1)= \rho(X_2) \qquad \text{whenever} \qquad \mathbb{P}_{X_1} = \mathbb{P}_{X_2}.
\end{equation*}
\end{definition}
Throughout the manuscript, we will use the term ``\emph{\licorm{}s}'' to refer to Law-Invariant COherent Risk Measures. Note that contrarily e.g.\@ to \cite[Section 6.3]{Shapiro2021}, we do not consider risk measures taking the value $+\infty$. We recall in addition that any coherent risk measure $\rho : \Lbb^p(\Omega,\R) \to\R$ with finite values is continuous (since it is convex), and admits the dual representation
\begin{equation} \label{eq:dual_repbis}
\rho(X)= \sup_{Y \in \mathfrak{Y}} \mathbb{E}[ XY ],
\end{equation}
where $\mathfrak{Y} \subset \Lbb^q(\Omega,\R)$ is a convex set of nonnegative random variables with unit expectation.


\paragraph{Generalized optimal transport.}

We now fix $p \in [1,+\infty]$ and denote by $q \triangleq p/(p-1)$ its conjugate exponent, with the usual conventions that $q= +\infty$ if $p=1$ and $q=1$ if $p=+\infty$.
Throughout the article, we will work with a subset $R \subset\mathcal{P}_q(\R)$, to which we associate the set $\mathcal{Y}_R \subset\R$ given by
\begin{equation*}
\mathcal{Y}_R \triangleq \bigcup_{r \in R} \text{supp}(r).
\end{equation*}
The set $R \subset\Pcal_q(\R)$ is the main ingredient involved in the definition of the generalized optimal transport problem we consider, and is assumed to satisfy all or part of the following assumptions.

\begin{taggedhyp}{\textnormal{(OT)}}
\label{hyp:OT} \hfill
\begin{enumerate}
\item[\textnormal{(i)}] The set $R \subset \mathcal{P}_q(\R)$ is bounded.
\item[\textnormal{(ii)}] It holds that $\mathcal{Y}_R \subset\R_+$ and $\mathbb{E}[r]=1$ for each $r \in R$.
\item[\textnormal{(iii)}] The set $R \subset \mathcal{P}_q(\R)$ is convex and closed for the weak-$^*$ topology.
\end{enumerate}
\end{taggedhyp}

In the remainder of the article, we will suppose that Assumption \ref{hyp:OT}-(i) is in force. Given $m \in\Pcal_p(\R)$, we consider the generalized optimal transport problem 
\begin{equation}
\label{eq:gal_ot}
\chi_R(m) \triangleq
\sup_{\pi \in \Pi(m,R)} \, \int_{\R^2} xy \dd \pi(x,y)
\end{equation}
where
\begin{equation*}
\Pi(m,R)
\triangleq
\Big\{ \pi \in \Pcal(\R^2) ~\,\textnormal{s.t.}~ \efrak^1_{\sharp} \pi= m ~~\text{and}~~ \efrak^2_{\sharp} \pi \in R \Big\}
\end{equation*}
with $\efrak^1(x,y) \triangleq x$ and $\efrak^2(x,y) \triangleq y$. Note that the integrability of the mapping $(x,y) \in \R^2 \mapsto xy \in \R$ will be justified in Lemma \ref{lemma:bound_transport} below. When $R \triangleq \{ r \}$ is a singleton, we shall simply write $\chi_r(m)$ and $\Pi(m,r)$ for the latter quantities, and note in particular that
\begin{equation*}
\label{eq:sup_formula}
\chi_R(m)= \sup_{r \in R} \, \chi_r(m).
\end{equation*}
By a direct application of H\"older's inequality, we have the following estimates.

\begin{lemma}[Elementary bounds on licorms]
\label{lemma:bound_transport}
For any $m \in \Pcal_p(\R)$ and any bounded subset $R \subset \Pcal_q(\R)$, it holds that
\begin{equation*}
\int_{\R^2} |xy| \dd \pi(x,y)
\leq \Mcal_p(m) \bigg( \sup_{r \in R} \, \Mcal_q(r) \bigg)
\end{equation*}
for any $\pi \in \Pi(m,R)$. In particular $| \chi_R(m) |
\leq \Mcal_p(m) \big( \sup_{r \in R} \, \Mcal_q(r) \big)$.
\end{lemma}

In the following proposition, we gather a few useful and well-known results from classical optimal transport theory, see e.g.\@ \cite{OTAM}. As before, we fix some $p \in [1,+\infty]$ and let $q \triangleq (p-1)/p$.

\begin{proposition}[Basic optimal transport results]
\label{proposition:ot}
Fix some $m \in \Pcal_p(\R)$ and let $\Xcal \triangleq \supp(m)$. Let also $\mathcal{Y} \subset\R$ be a closed subset, which is assumed to be bounded if $q= +\infty$, and fix an element $r \in \Pcal_q(\Ycal)$. Then, there exists a unique transport plan $\bar{\pi} \in \Pi(m,r)$ such that
\begin{equation*}
\chi_r(m) = \int_{\R^2} xy \dd \bar{\pi}(x,y).
\end{equation*}
Moreover, the latter is given explicitly by $( F_{m}^{-1} , F_{r}^{-1} )_{\sharp} \mathcal{L}^1_{\llcorner [0,1]}$ where $\mathcal{L}^1_{\llcorner [0,1]}$ stands for the restriction of the Lebesgue measure to $[0,1]$, so in particular
\begin{equation}
\label{eq:transport_formula}
\chi_r(m)=
\int_0^1 F_m^{-1}(t)F_r^{-1}(t) \dd t.
\end{equation}
Additionally, it holds that
\begin{equation} \label{eq:duality_ot}
\chi_r(m)
= \inf_{(f,g) \in K} \int_{\Xcal} f(x) \dd m(x) + \int_{\Ycal} g(y) \dd r(y) ,
\end{equation}
where $K \subset \Ccal_p^0(\Xcal) \times \Ccal_q^0(\mathcal{Y})$ is the function set defined by
\begin{equation} \label{eq:def_K}
K \triangleq \bigg\{
(f,g) \in \Ccal_p^0(\Xcal) \times \Ccal_q^0(\Ycal) ~\,\textnormal{s.t.}~ xy \leq f(x) + g(y) ~~ \text{for all $(x,y) \in \Xcal \times \mathcal{Y}$} \bigg\}.
\end{equation}
\end{proposition}

\begin{proof}
The existence of an optimal transport plan $\bar{\pi} \in\Pi(m,r)$ follows e.g.\@ from \cite[Theorem 4.1]{villani2009optimal}. Furthermore, by \cite[Theorem 5.10-(ii)]{villani2009optimal}, any such optimal plan $\bar{\pi}$ is monotone, which in the present context means that
\begin{equation*}
x_1y_2 + x_2y_1 \leq x_1 y_1  + x_2 y_2
\end{equation*}
for all pairs of elements $(x_1,y_1), (x_2,y_2) \in \textnormal{supp}(\bar{\pi})$, which implies in particular that $y_1 \leq y_2$ whenever $x_1 < x_2$.
Thanks e.g.\@ to \cite[Lemma 2.8]{OTAM}, such a plan is actually unique and given explicitly through the formula $( F_{m}^{-1} , F_{r}^{-1} )_{\sharp} \mathcal{L}^1_{\llcorner [0,1]}$.

To establish the duality formula \eqref{eq:duality_ot}, we first exhibit a pair $(a,b)$ that lies in $K$. To do so, it is enough to consider the two following cases.
\begin{itemize}
\item If $p \in (1,+\infty)$, we set $a(x) \triangleq |x|^p/p$ and $b(y) \triangleq |y|^q/q$ for all $(x,y) \in\Xcal\times\mathcal{Y}$. Then, Young's inequality entails that $xy \leq a(x) + b(y)$, for every $(x,y) \in \mathcal{X} \times \mathcal{Y}$.
\item If $p=1$, then $q = +\infty$ and $\Ycal$ is bounded under our working assumptions. We then set $a(x) \triangleq (\max_{y \in \Ycal} |y|) |x|$ and $b(y)= 0$. Clearly $(a,b)$ lies in $K$. The case $p=+\infty$ can be treated similarly.
\end{itemize}
At this stage, note that for any $(f,g) \in K$, there holds
\begin{equation*}
\chi_r(m)
\leq \sup_{\pi \in \Pi(m,r)}
\int_{\Xcal \times \Ycal}
\Big( f(x) + g(y) \Big) \dd \pi(x,y)
= \int_{\Xcal} f(x) \dd m(x)
+ \int_{\Ycal} g(y) \dd r(y).
\end{equation*}
This implies in particular that
\begin{equation} \label{eq:weak_duality_ot}
\chi_r(m)
\leq \inf_{(f,g) \in K}
\int_{\Xcal} f(x) \dd m(x)
+ \int_{\Ycal} g(y) \dd r(y) .
\end{equation}
Next, we define the cost function $c \colon \Xcal \times \Ycal \rightarrow \R$ by
\begin{equation*}
c(x,y) \triangleq a(x) + b(y) - xy,
\end{equation*}
for all $(x,y) \in \Xcal \times \Ycal$, and observe that it is nonnegative, since $(a,b) \in K$. Moreover, one can check that
\begin{equation} \label{eq:change1}
\chi_r(m)
= \int_{\Xcal} a(x) \dd m(x) +
\int_{\Ycal} b(y) \dd r(y)
- \inf_{\pi \in \Pi(m,r)} \int_{\Xcal \times \Ycal} c(x,y) \dd \pi(x,y), 
\end{equation}
so that upon letting $K_0 \subset \Ccal_b(\Xcal) \times \Ccal_b(\Ycal)$ be the set defined by 
\begin{equation*}
K_0 \triangleq \bigg\{ (\tilde{f},\tilde{g}) \in \Ccal_b(\Xcal) \times \Ccal_b(\Ycal) ~\,\textnormal{s.t.}~
\tilde{f}(x) + \tilde{g}(y) \leq c(x,y) ~~ \text{for all $(x,y) \in \Xcal \times \Ycal$} \bigg\},
\end{equation*}
we deduce from the usual Kantorovich duality theorem, see e.g.\@ \cite[Theorem 5.10]{villani2009optimal}, that
\begin{equation}\label{eq:change2}
\inf_{\pi \in \Pi(m,r)}
\int_{\Xcal \times \Ycal}
c(x,y) d \pi(x,y)
= \sup_{(\tilde{f},\tilde{g}) \in K_0} \int_{\Xcal} \tilde{f}(x) \dd m(x) + \int_{\Ycal} \tilde{g}(y) \dd r(y).
\end{equation}
Then, using the changes of variables $(f,g) = (a- \tilde{f},b- \tilde{g})$ while recalling the definition of the cost function $c : \Xcal \times \Ycal \to \R$, it can be checked that
\begin{equation*}
\begin{aligned}
K_1 & \triangleq \bigg\{
(f,g) \in K ~\,\textnormal{s.t.}~ (a-f) \in \Ccal_b(\Xcal) ~~\text{and}~~ (b-g) \in \Ccal_b(\Ycal) \bigg\} \\
& = \Big\{(a-\tilde{f},b-\tilde{g}) \in \Ccal^0_p(\Xcal) \times \Ccal^0_q(\Ycal) ~\,\textnormal{s.t.}~ (\tilde{f},\tilde{g}) \in K_0 \Big\},
\end{aligned}
\end{equation*}
which combined with \eqref{eq:change1}-\eqref{eq:change2} further yields
\begin{equation*}
\chi_r(m)= \inf_{(f,g) \in K_1} 
\int_{\Xcal} f(x) \dd m(x)
+ \int_{\Ycal} g(y) \dd r(y).
\end{equation*}
Finally, by resorting to \eqref{eq:weak_duality_ot} while noticing that $K_1 \subset K$, we finally get that
\begin{equation*}
\begin{aligned}
\chi_r(m) & \leq \;
\inf_{(f,g) \in K} 
\int_{\Xcal} f(x) \dd m(x) + \int_{\Ycal} g(y) \dd r(y)
\\
& \leq
\inf_{(f,g) \in K_1}
\int_{\Xcal} f(x) \dd m(x) + \int_{\Ycal} g(y) \dd r(y)
= \chi_r(m), 
\end{aligned}
\end{equation*}
which entails the duality formula \eqref{eq:duality_ot}.
\end{proof}


\section{Transport-based risk measures}
\label{section:ot-rm}

This section focuses on those risk measures $\rho_R : \Lbb^p(\Omega,\R) \to\R)$ associated with the value function $\chi_R : \Pcal_p(\R) \to \R$ of our generalized optimal transport problem \eqref{eq:gal_ot}, defined as 
\begin{equation}
\label{eq:def_rhoR}
\rho_R(X) := \chi_R(\mathbb{P}_X).
\end{equation}
In Theorem \ref{thm:kusuoka_converse} below, we show that $\rho_R$ is a licorm under Assumptions \ref{hyp:OT}-\textnormal{(i)} and \textnormal{(ii)}. In Lemma \ref{lemma:lip_chiR} and Corollary \ref{coro:lip_rhoR}, we investigate the Lipschitz continuity of $\chi_R$ and $\rho_R$.

\begin{theorem}[Licorms induced by optimal transport problems]
\label{thm:kusuoka_converse}
Suppose that $R \subset \mathcal{P}_q(\R)$ satisfies Assumptions \ref{hyp:OT}-\textnormal{(i)} and \textnormal{(ii)}.
Then, the mapping $\rho_R \colon \mathbb{L}^p(\Omega,\R) \to \R$ is a \licorm{}.
\end{theorem}

\begin{proof}
To show that the generalized optimal transport problem \eqref{eq:gal_ot} induces a licorm via \eqref{eq:def_rhoR}, one simply has to verify that the corresponding axioms listed in Definition \ref{def:coherence} above are satisfied. Note that the law-invariance property is trivially satisfied by construction.

\paragraph*{Translation invariance}

Let $X \in \Lbb^p(\Omega,\R)$ and fix an arbitrary $\pi \in \Pi(\P_X,R)$. Given $\alpha \in \R$, consider the map $\kappa_{\alpha}(x,y) \triangleq (x+\alpha,y)$, and observe that $\kappa^{\alpha}_{\sharp} \pi \in \Pi(\mathbb{P}_{X+\alpha},R)$. Thus
\begin{equation*}
\rho_R(X+ \alpha)
\geq \int_{\R^2} xy \dd (\kappa^{\alpha}_{\sharp} \pi) (x,y)
= \alpha + \int_{\R^2} xy \dd \pi(x,y),
\end{equation*}
where we used the fact that $\int_{\R^2} y \dd \pi(x,y) = 1$ as a direct consequence of Assumption \ref{hyp:OT}-(ii). Taking the supremum with respect to $\pi \in\Pi(m,R)$, we deduce that 
\begin{equation*}
\rho_R(X+\alpha) \geq \alpha + \rho_R(X).    
\end{equation*}
At this stage, changing $X$ for $X+\alpha$ while redefining $\kappa_{\alpha}(x,y) \triangleq (x-\alpha,y)$, we infer that $\rho_R(X) \geq \rho(X+\alpha)-\alpha$, which proves the desired property.

\paragraph*{Homogeneity}

Let $X \in \Lbb^p(\Omega,\R)$ and fix an arbitrary $\pi \in \Pi(\P_X,R)$. Given $\delta >0$, consider the map $\kappa^{\delta}(x,y) \triangleq (\delta x,y)$ and observe that $\kappa^{\delta}_{\sharp} \pi \in \Pi(\mathbb{P}_{\delta X}, R)$, whence 
\begin{equation*}
\rho_R(\delta X)
\geq \int_{\R^2} xy \dd (\kappa^{\delta}_{\sharp} \pi)(x,y)
= \delta \int_{\R^2} xy \dd \pi(x,y).
\end{equation*}
By taking the supremum with respect to  $\pi \in\Pi(\P_X,r)$, we deduce that 
\begin{equation*}
\rho_R(\delta X) \geq \delta \rho_R(X)    
\end{equation*}
Replacing now $X$ by $\delta X$ and considering instead $\kappa^{\delta}(x,y) \triangleq (x/\delta,y)$, we further obtain 
\begin{equation*}
\rho_R(X) \geq \tfrac{1}{\delta} \rho_R(\delta X), 
\end{equation*}
wherefore $\rho(\delta X) = \delta \rho_R(X)$ for each $\delta > 0$. This closes the proof upon noting that the latter equality obviously holds for $\delta=0$, since $\rho_R(0)=0$ by construction.

\paragraph*{Monotonicity} Take $X_1,X_2 \in \Lbb^p(\Omega,\R)$ such that $X_1 \leq X_2$ almost surely. Then, for every $\pi_1 \in \Pi(\P_{X_1},r)$, there exists by virtue of the gluing lemma (see e.g.\@ \cite[Lemma 5.3.2]{AGS}) a probability measure $\Bpi \in \Pcal(\R^3)$ such that 
\begin{equation*}
\efrak^{1,2}_{\sharp} \Bpi = \P_{(X_1,X_2)} \qquad \text{and} \qquad \efrak^{1,3}_{\sharp} \Bpi = \pi_1
\end{equation*}
Upon observing that $\pi_{2,3} \triangleq e^{2,3}_{\sharp} \Bpi \in \Pi(\P_{X_2},R)$, it follows that
\begin{equation*}
\rho_R(X_2)
\geq \int_{\R^2} x_2y \dd \pi_{2,3}(x_2,y)
= \int_{\R^3} x_2 y \dd \Bpi(x_1,x_2,y),
\end{equation*}
and since it both holds that $x_1 \leq x_2$ and $y\geq 0$ for $\Bpi$-almost every $(x_1,x_2,y) \in \R^3$, this implies
\begin{equation*}
\rho_R(X_2)
\geq  \int_{\R^3} x_1 y \dd \Bpi(x_1,x_2,y)
= \int_{\R^2} x_1y \dd \pi(x_1,y).
\end{equation*}
Since $\pi \in \Pi(\P_{X_1},R)$  was chosen arbitrarily, we conclude that $\rho_R(X_2) \geq \rho_R(X_1)$.

\paragraph*{Convexity} Let $X_1,X_2 \in \Lbb^p(\Omega,\R)$, choose $\theta \in [0,1]$ and set $X_{\theta} \triangleq (1-\theta) X_1 + \theta X_2$. Then, consider the map
\begin{equation*}
\kappa^{\theta}(x_1,x_2) \triangleq \big(x_1,x_2,(1-\theta)x_1 + \theta x_2 \big),
\end{equation*}
and define the measure $m_{\theta} \triangleq \kappa^{\theta}_{\sharp}\P_{(X_1,X_2)} \in\Pcal_p(\R^3)$. Next, take any $\pi_{\theta} \in \Pi(\P_{X_{\theta}},R)$ and apply the gluing lemma to build a measure $\Bpi_{\theta} \in\Pcal(\R^4)$ such that 
\begin{equation*}
\efrak^{1,2,3}_{\sharp} \Bpi_{\theta} = m_{\theta} \qquad \text{and} \qquad \efrak^{3,4}_{\sharp} \Bpi_{\theta} = \pi_{\theta},
\end{equation*}
where $\efrak^{1,2,3}(x,y,z,t) \triangleq (x,y,z)$ and $\efrak^{3,4}(x,y,z,t) = (z,t)$, and observe that
\begin{equation*}
\int_{\R^2} x_{\theta} y \dd \pi_{\theta}(x_{\theta},y)
= \int_{\R^4} x_{\theta} y \dd \Bpi_{\theta}(x_1,x_2,x_{\theta},y).
\end{equation*}
Recalling that $x_{\theta}= (1-\theta)x_1 + \theta x_2$ for $\Bpi_{\theta}$-almost every $(x_1,x_2,x_{\theta},y)$, we infer that
\begin{equation*}
\begin{aligned}
\int_{\R^2} x_{\theta} y \dd \pi_{\theta}(x_{\theta},y)
& = \int_{\R^4} \big( (1-\theta)x_1 + \theta x_2 \big) y \dd \Bpi_{\theta}(x_1,x_2,x_{\theta},y) \\
& \leq (1- \theta) \rho_R(X_1) + \theta \rho_R(X_2), 
\end{aligned}
\end{equation*}
and since $\pi_{\theta} \in\Pi(m_{\theta},R)$ was chosen arbitrarily, this yields by passing to the supremum that $\rho_R((1-\theta)X_1 + \theta X_2)  \leq (1-\theta) \rho_R(X_1) + \theta \rho_R(X_2)$, thereby closing the proof. 
\end{proof}

\begin{remark}[On our set of working assumptions]
\label{rem:convexity_rhoR}
Note that the convexity of $\rho_R \colon \Lbb^p(\Omega,\R) \to \R$ does not rely on Assumption \ref{hyp:OT}-\textnormal{(ii)}. The latter hypothesis is only needed to ensure that the risk measure is coherent. 
\end{remark}

In what follows, we prove basic regularity estimates on the value function of the generalized optimal transport problem and the induced licorm. 

\begin{lemma}[Lipschitz continuity of $\chi_R$]
\label{lemma:lip_chiR}
Suppose that $R \subset \mathcal{P}_q(\R)$ satisfies Assumption \ref{hyp:OT}-\textnormal{(i)}, and define
\begin{equation*}
L_R \triangleq
\sup_{r \in R} \, \Mcal_q(r).
\end{equation*}
Then $\chi_R : \Pcal_p(\R) \to \R$ is $L_R$-Lipschitz continuous for the Wasserstein distance, namely
\begin{equation*}
| \chi_R(m_2)-\chi_R(m_1) |
\leq L_R W_p(m_1,m_2),
\end{equation*}
for every $m_,m_2 \in \mathcal{P}_p(\R)$.
\end{lemma}

\begin{proof}
Fix some $\varepsilon>0$, take any $\pi_{1,2} \in \Pi(m_1,m_2)$ and let $\pi_2 \in \Pi(m_2,R)$ be such that
\begin{equation*}
\chi_R(m_2) \leq \int_{\R^2} xy \dd \pi_2(x,y) + \varepsilon.
\end{equation*}
Then, by the gluing lemma, there exists a measure $\Bpi \in \mathcal{P}(\R^3)$ such that 
\begin{equation*}
\efrak^{1,2}_{\sharp} \Bpi = \pi_{1,2} \qquad \text{and} \qquad \efrak^{2,3}_{\sharp}\Bpi = \pi_2
\end{equation*}
Since by construction it clearly holds that $\mathfrak{e}^{1,3}_{\sharp} \Bpi \in \Pi(m_1,R)$, we may compute
\begin{align*}
|\chi_R(m_2) - \chi_R(m_1)| &\le \bigg| \int_{\R^2} x_2 y \; \mathrm{d} \pi_2(x_2,y) - \int_{\R^3} x_1 y \; \mathrm{d} (\mathfrak{e}^{1,3}_{\sharp} \Bpi)(x_1,y)\bigg| + \varepsilon \\
& \leq \int_{\R^3} |x_2 - x_1| |y| \; \mathrm{d} \Bpi(x_1,x_2,y) + \varepsilon \\
& \leq \Mcal_q(\efrak^3_{\sharp} \Bpi) \bigg( \INTDom{|x_1-x_2|^p}{\R^3}{\Bpi(x_1,x_2,y)} \bigg)^{1/p} + \varepsilon \\
& \leq L_R \bigg( \INTDom{|x_2-x_1|^p}{\R^2}{\pi_{1,2}(x_1,x_2,y)} \bigg)^{1/p} + \varepsilon
\end{align*}
and we conclude from the arbitrariness of $\pi_{1,2} \in \Pi(m_1,m_2)$ and $\varepsilon > 0$.
\end{proof}

\begin{corollary}[Lipschitz continuity of $\rho_R$]
\label{coro:lip_rhoR}
Suppose that $R \subset \mathcal{P}_q(\R)$ satisfies Assumption \ref{hyp:OT}-\textnormal{(i)}. Then $\rho_R : \Lbb^p(\Omega,\R) \to \R$ is $L_R$-Lipschitz continuous.
\end{corollary}

\begin{proof}
Given $X_1,X_2 \in \mathbb{L}^p(\Omega,\R)$, recall that the joint distribution of the couple $(X_1,X_2)$ is given by $\P_{(X_1,X_2)} := (X_1,X_2)_{\sharp} \P$. It follows then from Lemma \ref{lemma:lip_chiR} that
\begin{align*}
|\rho_R(X_2)-\rho_R(X_1)|
\leq {} & L_R W_p(\mathbb{P}_{X_1},\mathbb{P}_{X_2}) \\
\leq {} & L_R \bigg( \INTDom{|x_1-x_2|^2}{\R^2}{\P_{(X_1,X_2)}(x,y)} \bigg) \\
= {} & L_R \| X_2 - X_1 \|_{\mathbb{L}^p(\Omega,\R)},
\end{align*}
as was to be proved.
\end{proof}

We close this section by showing that aversity to risk is elementarily built in licorms.  

\begin{lemma}[Aversity to risk of licorms]
\label{lemma:expect_is_lower_bound}
Suppose that $R \subset \mathcal{P}_q(\R)$ satisfies Assumptions \ref{hyp:OT}-\textnormal{(i)} and \textnormal{(ii)}.
Then it holds that
\begin{equation}
\label{eq:expect_is_lower_bound}
\rho_R(X) \geq \mathbb{E}[X],
\end{equation}
for every $X \in \mathbb{L}^p(\Omega,\R)$.
\end{lemma}

\begin{proof}
Let $X \in \mathbb{L}^p(\Omega,\R)$ and $r \in R$. Setting $m \triangleq \mathbb{P}_X$ and $\pi \triangleq m \times r$, it clearly holds that $\pi \in \Pi(m,R)$, so that 
\begin{equation*}
\rho_R(X)
= \chi_R(m)
\geq \int_{\R^2} xy \dd \pi(x,y)
= \bigg( \int_{\R} x \dd m(x) \bigg)
\bigg( \int_{\R} y \dd r(y) \bigg)
= \mathbb{E}[X],
\end{equation*}
as was to be proved.
\end{proof}


\section{Optimal transport characterization of \licorm{}s}
\label{section:characterization}

In this section, we prove a partial yet sharp converse to Theorem \ref{thm:kusuoka_converse}. Specifically, we show that any \licorm{} defined on a nonatomic probability space can be written in the form introduced in \eqref{eq:def_rhoR}. Note that when $p=+\infty$, we must assume that the risk measure also satisfies the so-called \textit{Fatou property}, whose definition is borrowed from \cite{delbaen2002coherent} and recalled below.

\begin{definition}[Fatou property]
A risk measure $\rho \colon \mathbb{L}^\infty(\Omega,\R) \rightarrow \R$ is said to satisfy the \emph{Fatou property} if for any bounded sequence of random variables $(X_n)_{n \in \mathbb{N}} \subset\mathbb{L}^\infty(\Omega,\R)$ converging in probability to $X \in \mathbb{L}^\infty(\Omega,\R)$, it holds that $\rho(X) \leq \underset{n \to +\infty}{\liminf} \, \rho(X_n)$.
\end{definition}

\begin{remark}[A simple criterion entailing the Fatou property] 
In \cite[Theorem 2.1]{jouini2006law}, it is shown that every licorm $\rho: \Lbb^{\infty}(\Omega,\R) \to \R$ satisfies the Fatou property as soon as $(\Omega,\Acal,\P)$ is a standard probability space.
\end{remark}

We are now ready to state and prove the main result of this section. 

\begin{theorem}[Optimal transport characterization of licorms]
\label{thm:kusuoka}
Suppose that the probability space $(\Omega,\mathcal{A},\mathbb{P})$ is nonatomic and let $\rho \colon \Lbb^p(\Omega,\R) \to \R$ be a \licorm{}, which we assume satisfies the Fatou property if $p=+\infty$.
Then there exists a subset $R \subset\mathcal{P}_q(\R)$ satisfying Assumptions \ref{hyp:OT}-\textnormal{(i)} and \textnormal{(ii)} such that
\begin{equation*}
\rho(X)= \chi_R(\mathbb{P}_X)
\end{equation*}
for each $X \in\Lbb^p(\Omega,\R)$.
\end{theorem}

The theorem will be obtained as an immediate corollary of Proposition \ref{prop:kusuoka_refined} below. As briefly explained in the introduction, the representation of a \licorm{} through an optimal transportation problem is intimately related to the representation of coherent risk measures as support functions, recalled in the following lemma.

\begin{lemma}[Dual representation of coherent risk measures]
\label{lemma:licorm_support}
Let $\rho \colon \mathbb{L}^p(\Omega,\R) \rightarrow \R$ be a coherent risk measure, which is assumed to satisfy the Fatou property if $p= +\infty$. Then there exists a convex bounded subset $\mathfrak{Y} \subset\mathbb{L}^q(\Omega,\R)$ containing only nonnegative random variables whose expectation is equal to 1, such that
\begin{equation}
\label{eq:rm_support}
\rho(X)
=
\sup_{Y \in \mathfrak{Y}}
\mathbb{E} \big[ XY \big].
\end{equation}
\end{lemma}

\begin{proof}
This result is demonstrated for $p<+\infty$ in \cite[Theorems 6.4 and 6.6]{Shapiro2021}. In the case in which $p=\infty$ and $\rho$ satisfies the Fatou property, we refer the reader to \cite[Theorem 3.2]{delbaen2002coherent}.
\end{proof}

\begin{proposition}[A refined characterization]
\label{prop:kusuoka_refined}
Let $\rho \colon \mathbb{L}^p(\Omega,\R) \rightarrow \R$ be a coherent risk measure, which is assumed to satisfy the Fatou property if $p= +\infty$. Let $\mathfrak{Y} \subset\mathbb{L}^q(\Omega,\R)$ be a bounded subset containing only nonnegative random variables whose expectation is equal to 1 such that \eqref{eq:rm_support} holds, and define
\begin{equation*}
R \triangleq \Big\{ \mathbb{P}_Y \in \Pcal_q(\R) ~\,\textnormal{s.t.}~ Y \in \mathfrak{Y} \Big\}.
\end{equation*}
Then $R \subset\Pcal_q(\R)$ satisfies Assumptions \ref{hyp:OT}-\textnormal{(i)} and \textnormal{(ii)}, and it holds that
\begin{equation}
\label{eq:kusuoka_plus1}
\rho(X)
\leq \rho_R(X)
\end{equation}
for each $X \in \mathbb{L}^p(\Omega,\R)$.
Moreover, if the probability space $(\Omega,\mathcal{A},\P)$ is nonatomic, one has that
\begin{equation}
\label{eq:kusuoka_plus2}
\sup_{
\begin{subarray}{c}
X' \in \mathbb{L}^p(\Omega,\R) \\
\mathbb{P}_{X'} = \mathbb{P}_X \end{subarray}} \rho(X')= \rho_R(X),
\end{equation}
for each $X \in \mathbb{L}^p(\Omega,\R)$. In particular, if $\rho : \Lbb^p(\Omega,\R) \to \R$ is law-invariant, then $\rho= \rho_R$.
\end{proposition}

\begin{proof}
Let $\rho : \Lbb^p(\Omega,\R) \to \R$ be a coherent risk measure and let $\mathfrak{Y} \subset L^q(\Omega,\R)$ be as in Lemma \ref{lemma:licorm_support}.
Since the latter is a bounded set containing only nonnegative random variables with expectation equal to 1, the set $R \subset \Pcal_q(\R)$ defined above clearly satisfies Assumptions \ref{hyp:OT}-\textnormal{(i)} and \textnormal{(ii)}.
Let now $X \in \mathbb{L}^p(\Omega,\R)$ and $Y \in \mathfrak{Y}$, and define $\pi \triangleq \mathbb{P}_{(X,Y)} \in \Pcal(\R^2)$. Then $\pi \in \Pi(\P_X,R)$ by construction, and 
\begin{equation*}
\mathbb{E}\big[ XY \big]
= \int_{\R^2} xy \dd \pi(x,y)
\leq \chi_R(\P_X)
= \rho_R(X).
\end{equation*}
Maximizing the left-hand side with respect to $Y \in \mathfrak{Y}$ while using Lemma \ref{lemma:licorm_support}, we obtain \eqref{eq:kusuoka_plus1}.

Let us assume now that the underlying probability space is nonatomic. To prove \eqref{eq:kusuoka_plus2}, we leverage again Lemma \ref{lemma:licorm_support}, which yields that
\begin{equation}
\label{eq:kusuoka_step1}
\sup_{X' \sim X} \rho(X')
=
\sup_{
X' \sim X} \sup_{Y \in \mathfrak{Y}}
\mathbb{E}\big[ X'Y \big]
=
\sup_{Y \in \mathfrak{Y}}
\sup_{
X' \sim X} 
\mathbb{E}\big[ X'Y \big],
\end{equation}
where we write $X' \sim X$ to mean that the random variable $X' \in\Lbb^p(\Omega,\R)$ has the same law as $X \in\Lbb^p(\Omega,\R)$. Next, we claim that
\begin{equation}
\label{eq:ot_interpretation}
\sup_{
X' \sim X} 
\mathbb{E}\big[ X'Y \big]
=
\int_0^1
F_X^{-1}(t)F_Y^{-1}(t) \dd t
=
\chi_{\mathbb{P}_Y}(\mathbb{P}_X).
\end{equation}
The first equality is demonstrated in \cite[Lemma 4.60]{follmer2011stochastic} while the second one immediately follows from \eqref{eq:transport_formula} in Proposition \ref{proposition:ot}.
Combining \eqref{eq:kusuoka_step1} and \eqref{eq:ot_interpretation}, we finally obtain that
\begin{equation*}
\sup_{
X' \sim X} \rho(X')
=
\sup_{Y \in \mathfrak{Y}}
\chi_{\mathbb{P}_Y}(\mathbb{P}_X)
= \chi_R(\mathbb{P}_X)
= \rho_R(X),
\end{equation*}
as was to be demonstrated.
\end{proof}

\begin{remark}[Validity of Theorem \ref{thm:kusuoka} for equiprobable discrete spaces]
The assumption in the statement of Theorem \ref{thm:kusuoka} and Proposition \ref{prop:kusuoka_refined} that $(\Omega,\Acal,\P)$ be nonatomic is only used for proving \eqref{eq:ot_interpretation}, as it is required in \cite[Lemma 4.60]{follmer2011stochastic}. However, it should be stressed that \eqref{eq:ot_interpretation} is also true when $(\Omega,\mathcal{A},\mathbb{P})$ is a finite set made of equiprobable events, see \cite[Lemma 6.25]{Shapiro2021}.
\end{remark}

\begin{remark}[Failure of Theorem \ref{thm:kusuoka} for non equiprobable discrete spaces]
In this remark, we exhibit an atomic probability space $(\Omega,\Acal,\P)$ for which the characterization result of Theorem \ref{thm:kusuoka} fails dramatically. Take $\Omega \triangleq \{ \omega_0,\omega_1 \}$ with $\mathbb{P}(\{\omega_0\})= 1/3$, and consider
\begin{equation*}
\rho(X) \triangleq X(\omega_1).
\end{equation*}
Clearly $\rho \colon \Lbb^p(\Omega,\R) \to \R$ is a coherent risk measure, and is also law-invariant since in our context, two different random variables necessarily have different probability distributions. Indeed, let $X \in \Lbb^p(\Omega,\R)$ be a random variable with probability distribution $m \triangleq \P_X$. Then, either $m= \delta_a$, in which case $X(\omega_0)= X(\omega_1)=a$, or $m = \frac{1}{3} \delta_a + \frac{2}{3} \delta_b$ and then $X(\omega_0)= a$ and $X(\omega_1)= b$. Hence, the random variable is fully determined by its law. Consider now the random variable defined by $X(\omega_0)= 1$ and $X(\omega_1)= 0$, and note that
\begin{equation*}
\rho(X)= 0 < \tfrac{1}{3} = \mathbb{E}[X].
\end{equation*}
This shows that the inequality \eqref{eq:expect_is_lower_bound} from Lemma \ref{lemma:expect_is_lower_bound} fails to hold, and thus there exists no subset $R \subset \Pcal_q(\R)$ satisfying Assumptions \ref{hyp:OT}-\textnormal{(i)} and \textnormal{(ii)} such that $\rho= \rho_R$.
\end{remark}

\begin{remark}[Comparison with Kusuoka's theorem]
\label{remark:kusuoka_comp}
As explained in the introduction, Kusuoka's theorem and our optimal-transport representation of licorms are strongly connected. We first mention that the proof of Theorem \ref{thm:kusuoka} relies on similar arguments to those developed in the proof of Kusuoka's theorem. We highlight here how our representation theorem involving an optimal transport problem can be recovered from the original representation due to Kusuoka.

As a reminder, Kusuoka's theorem \cite{kusuoka2001law} states that any licorm $\rho \colon \mathbb{L}^\infty(\Omega,\R) \rightarrow \R$ satisfying the Fatou property can be written as
\begin{equation}
\label{eq:kusuoka_rep}
\rho(X) = \sup_{\mu \in \mathcal{M}}
\int_0^1 \mathrm{CV@R}_{\beta}(X) \dd \mu(\beta),
\end{equation}
where $\mathcal{M} \subset\Pcal([0,1))$ is a subset of probability measures. Note that this result also holds for $p \in [1,+\infty)$, see e.g. \cite[Theorem 6.24]{Shapiro2021}.
We explain now how our  optimal transport representation may be partially recovered from \eqref{eq:kusuoka_rep}.
The key idea is to show that every convex combination of $(\mathrm{CV@R}_{\beta})_{\beta \in[0,1)}$ can be put in the form \eqref{eq:def_rhoR} with $R:=\{r\}$, and therefore expressed through a standard transport problem. We stress that this is a known fact, and refer the reader e.g. to\@ \cite{ekeland2012comonotonic,ennaji2024robust}. Taking an element $\mu \in\Pcal([0,1))$ and using the representation formula for the conditional value at risk provided e.g. in \cite[Theorem 6.2]{Shapiro2021}, we obtain that
\begin{align*}
\int_0^1 \mathrm{CV@R}_{\beta}(X) \dd \mu(\beta)
= {} &
\int_0^1 \frac{1}{1-\beta} \int_{\beta}^1 F_X^{-1}(t) \dd t \dd \mu(\beta) \\
= {} &
\int_0^1 \Big( \int_{0}^t \frac{1}{\beta} \dd \mu(\beta) \Big)
F_X^{-1}(t) \dd t \\
= {} &
\int_0^1 \psi_{\mu}(t) F_X^{-1}(t) \dd t,
\end{align*}
where $\psi_{\mu} \colon t \in [0,1] \mapsto \int_{0}^t \frac{1}{1-\beta} \dd \mu(\beta) \in \R_+$. The function $\psi_{\mu}$ is obviously nondecreasing and satisfies $\psi_{\mu}(0)=0$. Besides, it can be made right-continuous up to a modification on a set of measure zero, while remaining nonnegative. Upon setting $r_{\mu} \triangleq (\psi_{\mu})_\sharp \mathcal{L}^1_{\llcorner [0,1]}$, we easily check that $F_{r_\mu}^{-1}= \psi_{\mu}$. Moreover, it holds that $r_{\mu}\in\Pcal(\R_+)$ by construction, and we then deduce from Proposition \ref{proposition:ot} that
\begin{equation*}
\int_0^1 \mathrm{CV@R}_{\beta}(X) \dd \mu(\beta)
= \int_0^1 F_{r_\mu}^{-1}(t) F_X^{-1}(t) \dd t
= \rho_{r_\mu}(X), 
\end{equation*}
so that 
\begin{equation*}
\rho(X)= \sup_{\mu \in \mathcal{M}} \rho_{r_{\mu}}(X)= \rho_R(X)
\end{equation*}
with $R \triangleq \bigcup_{\mu \in \mathcal{M}} \{ r_{\mu} \}$.
There now remains to check whether the latter set satisfies Assumptions \ref{hyp:OT}-(i) and (ii). We can verify quite straightforwardly that $\mathbb{E}[ r_{\mu}]=1$ for every $\mu \in \mathcal{M}$. To this purpose, observe that for the constant random variable $X \equiv 1$, one has that $F_X^{-1}(t)=1$ for every $t>0$, and so
\begin{equation*}
\rho_{r_{\mu}}(X) = 1= \int_0^1 \mathrm{CV@R}_{\beta}(1) \dd \mu(\beta) = \rho_{r_\mu}(1)= \int_0^1 F_{r_\mu}^{-1}(t) \dd t
= \int_0^1 \psi_{\mu}(t) \dd t
= \int_0^{+\infty} y \dd r_{\mu}(y),
\end{equation*}
as announced. At this stage, however, it is not clear whether the boundedness of $R\subset\mathcal{P}_q(\R)$ can be easily deduced from Kusuoka's theorem, hence the partial recovery of our own result.
\end{remark}


\section{Duality formulas}
\label{section:duality}

In this section, we prove a general duality formula for risk measures of the form $\rho_R$, under the assumption that $R \subset\Pcal_q(\R)$ is a convex set. In this context, we fix a measure $m \in\Pcal_p(\R)$ and denote by $\Xcal \triangleq  \supp(m)$, and suppose that Assumption \ref{hyp:OT}-\textnormal{(iii)} is in force throughout the section. We point out that Assumption \ref{hyp:OT}-\textnormal{(ii)} will not be used anywhere in our subsequent developments, so in particular, $\rho_R$ may possibly not be a coherent risk-measure.

We recall that $\mathcal{Y}_R \triangleq \bigcup_{r \in R} \, \text{supp}(r)$, and define the support function $\sigma_R \colon \Ccal_q^0(\Ycal_R) \to \R$ of the set $R \subset\Pcal_q(\R)$ as
\begin{equation*}
\sigma_R(g) \triangleq \sup_{r \in R} \, \int_{\Ycal_R} g(y) \dd r(y),
\end{equation*}
for each $g \in \Ccal_q^0(\Ycal_R)$.
In what follows, we will extensively work with the function set introduced in \eqref{eq:def_K} above in the particular case in which $\mathcal{Y} \triangleq \mathcal{Y}_R$, namely
\begin{equation}
\label{eq:Kdef}
K = \bigg\{
(f,g) \in \Ccal_p^0(\Xcal) \times \Ccal_q^0(\Ycal_R) ~\,\textnormal{s.t.}~ xy \leq f(x) + g(y) ~~ \text{for all $(x,y) \in \Xcal \times \Ycal_R$} \bigg\}.
\end{equation}
We are now ready to introduce our first dual problem, given by
\begin{equation} \label{eq:first_dualpb}
\inf_{(f,g) \in K}
\Jcal(f,g) \triangleq
\inf_{(f,g) \in K} \int_{\Xcal} f(x) \dd m(x)
+ \sigma_R(g). 
\end{equation}
In the sequel, we will also consider another dual problem, in which $f \in\Ccal^0_p(\Xcal)$ is replaced by the Fenchel conjugate of $g \in\Ccal_q^0(\Ycal_R)$, see e.g.\@ \cite[Chapter 6]{Hiriart1996}, defined by
\begin{equation*}
g^*(x) \triangleq \sup_{y \in \R} \, xy - g(y) \in \R \cup \{ +\infty \}.
\end{equation*}
for each $x \in\Xcal$. To make sense of this new dual problem, we first need to justify that $g^*$ is $m$-integrable, in the sense that its integral is well defined. To this end, we fix $y_0 \in \Ycal_R$, and note that $g^*(x) \geq xy_0- g(y_0)$ for all $x \in \Xcal$, which implies that
\begin{equation*}
\int_{\Xcal} \Big( g^*(x) - (xy_0-g(y_0)) \Big) \dd m(x) \in \R_+ \cup \{ +\infty \}
\end{equation*}
is well-defined as the integral of a nonnegative function. Similarly, observe that 
\begin{equation*}
\int_{\Xcal} (xy_0 - g(y_0)) \dd m(x)
\end{equation*}
is finite, since every affine functions lies in $\mathcal{C}_p^0(\Xcal,\R)$ regardless of the value of $p \in[1,+\infty]$ (recall in particular that $\Xcal = \supp(m)$ is assumed to bes bounded if $p= +\infty$). Whence, the integral 
\begin{equation*}
\INTDom{g^*(x)}{\Xcal}{m(x)} =  \int_{\Xcal} \Big( g^*(x) - (xy_0-g(y_0)) \Big) \dd m(x) + \int_{\Xcal} (xy_0 - g(y_0)) \dd m(x), 
\end{equation*}
is well-defined in $\R \cup\{+\infty\}$, and it follows easily from the previous identity that 
\begin{equation} \label{eq:equi_def_int}
\int_{\Xcal} g^*(x) \dd m(x)= +\infty
\quad \text{if and only if} \quad
\int_{\Xcal} |g^*(x)| \dd m(x)= +\infty.
\end{equation}
Our second dual problem is then defined as
\begin{equation} 
\label{eq:second_dualpb}
\inf_{g \in \Ccal_q^0(\Ycal_R)}
\tilde{\Jcal}(g) \triangleq
\int_{\Xcal} g^*(x) \dd m(x) 
+ \sigma_R(g).
\end{equation}

\begin{theorem}[Duality formulas]
\label{theo:duality}
Let $m \in \Pcal_p(\R)$ and denote by $\Xcal \triangleq \supp(m)$.
Suppose that $R \subset\Pcal_q(\R_+)$ satisfies Assumption \ref{hyp:OT}-\textnormal{(i)}.
Then, it holds that
\begin{equation} \label{eq:weakDuality1}
\chi_R(m) \leq \inf_{g \in \Ccal_q^0(\Ycal_R)} \tilde{\Jcal}(g)
\leq \inf_{(f,g) \in K} \Jcal(f,g).
\end{equation}
In addition, if Assumption \ref{hyp:OT}-\textnormal{(iii)} holds, then both inequalities become equalities, namely
\begin{equation} \label{eq:StrongDuality}
\chi_R(m) = \inf_{g \in \Ccal_q^0(\Ycal_R)} \tilde{\Jcal}(g)
= \inf_{(f,g) \in K} \Jcal(f,g).
\end{equation}
In particular if $(f,g) \in K$ is a solution of \eqref{eq:first_dualpb}, then $g \in\Ccal_q^0(\Ycal)$ is a solution of \eqref{eq:second_dualpb}.
\end{theorem}

\begin{remark}[Duality formula for the risk measure]
The previous theorem translates as follows for the risk measure $\rho_R \colon \Lbb^p(\Omega,\R) \to \R$. Under Assumption \ref{hyp:OT}-\textnormal{(i)}, it holds that
\begin{equation*}
\rho_R(X) \leq \inf_{g \in \mathcal{C}_q^0(\mathcal{Y}_R)} \mathbb{E}[ g^*(X) ] + \sigma_R(g(X))
\leq \inf_{(f,g) \in K} \mathbb{E}[ f(X)] + \sigma_R(g(X)).
\end{equation*}
If in addition Assumption \ref{hyp:OT}-\textnormal{(iii)} is satisfied, then both inequalities become equalities.
\end{remark}

\begin{remark}[Applications to stochastic programming]
Duality formula are known to be of key relevance to investigate stochastic programs in general, which in our context take the form
\begin{equation}
\label{eq:stoch_prog0}
\inf_{u \in \mathcal{U}} \rho_R(X[u]),
\end{equation}
where $\mathcal{U}$ is a given feasible set and $X \colon u \in \mathcal{U} \mapsto X[u] \in \mathbb{L}^p(\Omega,\R)$ is a correspondence between the input decision variable and the observed random variable. If $R \subset\Pcal_q(\R)$ satisfies Assumptions \ref{hyp:OT}-(i) and (iii), then the above problem is equivalent to
\begin{equation}
\label{eq:stoch_prog2bis}
\inf_{
\begin{subarray}{c}
g \in \mathcal{C}_q^0(\mathcal{Y}_R) \\
u \in \mathcal{U} \end{subarray}}
\mathbb{E} \big[ g^*(X[u]) \big] + \sigma_R(g(X)),
\end{equation}
which is very likely to be easier to investigate from a numerical point of view, as it is a joint minimization problem and not a min-max problem.
In addition, suppose that \eqref{eq:stoch_prog0} admits a solution $\bar{u} \in\Ucal$. Then, if the dual problem corresponding to $\mathbb{P}_{X[\bar{u}]} \in\Pcal_p(\R)$ also has a solution $\bar{g} \in\Ccal^0_q(\Ycal_R)$, then the pair $(\bar{g},\bar{u})$ is a solution of \eqref{eq:stoch_prog2bis}, a problem for which it may be easier to formulate optimality conditions, as the risk measure -- which may be nonsmooth -- does not appear explicitly. This motivates the study conducted at the end of the section concerning the existence of dual solutions.
\end{remark}

\begin{proof}[Proof of Theorem \ref{theo:duality}]
We split the proof of this theorem into two steps. In Step 1, we start by establishing the chain of weak duality inequalities displayed in \eqref{eq:weakDuality1}, whereas in Step 2 we prove the strong duality identities \eqref{eq:StrongDuality} by showing that $\chi_R(m) = \inf_{(f,g) \in K} \Jcal(f,g)$. 

\paragraph*{Step 1 -- Weak duality formulas.} 

We begin by proving the first inequality in \eqref{eq:weakDuality1}. To do so, let $g \in \Ccal_q^0(\Ycal_R)$, fix some $\pi \in \Pi(m,R)$ and denote by $r \triangleq \efrak^2_{\sharp} \pi$. We claim that
\begin{equation} \label{eq:claim_weak}
\int_{\Xcal \times \Ycal_R} xy \dd \pi(x,y)
\leq \int_{\Xcal} g^*(x) \dd m(x)
+ \int_{\Ycal_R} g(y) \dd r(y).
\end{equation}
Indeed, if $\int_{\Xcal} g^*(x) \dd m(x)= +\infty$, the inequality is trivially satisfied. Otherwise, $g^*$ is necessarily $m$-summable by the equivalence \eqref{eq:equi_def_int}, and \eqref{eq:claim_weak} follows from the Fenchel-Young inequality
\begin{equation*}
xy \leq g(x) + g^*(y),
\end{equation*}
which holds for all $(x,y) \in\Xcal \times\Ycal_R$. At this stage, observe that \eqref{eq:claim_weak} entails in particular that
\begin{equation*}
\int_{\Xcal \times \Ycal_R}
xy \dd \pi(x,y) \leq \int_{\Xcal} g^*(x) \dd m(x) + \sigma_R(g).
\end{equation*}
Taking the supremum with respect to $\pi \in \Pi(m,R)$ and then the minimum with respect to $g \in \Ccal_q^0(\Ycal_R)$, we further obtain
\begin{equation*}
\chi_R(m) \leq
\inf_{g \in \Ccal_q^0(\Ycal_R)}
\bigg\{ \int_{\Xcal} g^*(x) \dd m(x)
+ \sigma_R(g) \bigg\}.
\end{equation*}
Let us now prove the second inequality in \eqref{eq:weakDuality1}. To this end, let $(f,g) \in K$ and note that for any $x \in \Xcal$, there holds
\begin{equation*}
g^*(x) = \, \sup_{y \in \Ycal_R} xy- g(y) \leq f(x), 
\end{equation*}
whence
\begin{equation*}
\int_{\Xcal} g^*(x) \dd m(x) \leq \int_{\Xcal} f(x) \dd m(x)    
\end{equation*}
and consequently $\tilde{\Jcal}(g) \leq \Jcal(f,g)$ for every $(f,g) \in K$. Assume lastly that both problems \eqref{eq:first_dualpb} and \eqref{eq:second_dualpb} have the same value (which will be verified next under Assumption \ref{hyp:OT}-(iii)) and that $(f,g) \in K$ is a solution of \eqref{eq:first_dualpb}. Since we have shown that $\tilde{\Jcal}(g) \leq \Jcal(f,g)$, it necessarily follows that $g \in\Ccal^0_q(\Ycal_R)$ is a solution of \eqref{eq:second_dualpb}.

\paragraph*{Step 2 -- Strong duality formulas.} 

We next prove the chain of equalities in \eqref{eq:StrongDuality} under Assumption \ref{hyp:OT}-(iii). Note that as a consequence of \eqref{eq:weakDuality1}, it suffices to show that $\chi_R(m)= \inf_{(f,g) \in K} \Jcal(f,g)$. To this end, we first establish the duality formula
\begin{equation} \label{eq:altFor}
\chi_R(m)
= \sup_{(m',r') \,\in\, \Ccal_p^0(\Xcal)^* \times \Ccal_q^0(\Ycal_R)^*}
\Big\{
- \iota_{R_m}(m',r') - \sigma_K(-m',-r') \Big\},
\end{equation}
where $R_m \triangleq \{ m \} \times R$ and $\iota_{R_m} \colon \Ccal_q^0(\Xcal)^* \times\Ccal_q^0(\Ycal_R)^* \to[0,+\infty]$ is the convex indicator function
\begin{equation*}
\iota_{R_m}(m',r') \triangleq 
\left\{
\begin{aligned}
& 0 ~~ &\text{if $(m',r') \in R_m$}, \\
& +\infty ~~ & \text{otherwise}. 
\end{aligned}
\right.
\end{equation*}
Next we let $\sigma_K \colon \Ccal_p^0(\Xcal)\times\Ccal_q^0(\Ycal_R) \to \R \cup\{+\infty\}$ be the support function of the set $K$, defined as
\begin{equation*}
\sigma_K(m',r') \triangleq \sup_{(f,g) \in K} \, \INTDom{\Big( f(x) + g(y) \Big)}{\Xcal \times \Ycal_R}{(m' \times r')(x,y)}
\end{equation*}
for each $(m',r') \in \Ccal_p^0(\Xcal)^* \times \Ccal_q^0(\Ycal_R)^*$, where we used the usual duality identification $(\Ccal_p^0(\Xcal) \times \Ccal_q^0(\Ycal_R))^* \simeq \Ccal_p^0(\Xcal)^* \times \Ccal_q^0(\Ycal_R)^*$. By Proposition \ref{proposition:ot}, we know that
\begin{equation*}
\chi_{r'}(m')
= - \sup_{(f,g) \in K} \bigg\{- \int_{\Xcal} f(x) \dd m'(x) - \int_{\Ycal_R} g(y) \dd r'(y) \bigg\}
= - \sigma_K(-m',-r')
\end{equation*}
for any $(m',r') \in R_m$, wherefore
\begin{equation} \label{eq:someEq}
\chi_R(m)
= \sup_{(m',r') \in R_m} \chi_{r'}(m')
= \sup_{(m',r') \in R_m} - \sigma_K(-m',-r')
\end{equation}
from which \eqref{eq:altFor} follows. Next we turn our attention back to the dual problem \eqref{eq:first_dualpb}, and begin by observing that we have the following equalities
\begin{equation*}
\begin{aligned}
\inf_{(f,g) \in K} \int_{\Xcal} f(x) \dd m(x) + \sigma_{R} (g)
& =  \inf_{(f,g) \in K} \sup_{(m',r') \in R_m} \int_{\Xcal} f(x) \dd m'(x) + \int_{\Ycal_R} g(y) \dd r'(y)\\
& = \inf_{(f,g) \in K} \sigma_{R_m}(f,g) \\
& = \inf_{(f,g) \in \Ccal_p^0(\Xcal) \times \Ccal_q^0(\Ycal_R)} \sigma_{R_m}(f,g) + \iota_K(f,g).
\end{aligned}
\end{equation*}
Remarking that $\sigma_{R_m} : \Ccal^0_p(\Xcal) \times\Ccal^0_q(\Ycal_R) \to[0,+\infty]$ has full domain whereas $\iota_K : \Ccal^0_p(\Xcal) \times\Ccal^0_q(\Ycal_R) \to[0,+\infty]$ is convex and proper (the non-emptiness of $K$ was verified in the proof of Proposition \ref{proposition:ot}), we may apply Fenchel-Rockafellar's duality theorem (see e.g.\@ \cite{Rockafellar1966}) to obtain
\begin{equation}
\label{eq:someothereq}
\inf_{(f,g) \in K} \int_{\Xcal} f(x) \dd m(x) + \sigma_{R} (g)
\, = \sup_{(m',r') \,\in\, \Ccal_p^0(\Xcal)^* \times \Ccal_q^0(\Ycal_R)^*} \Big\{
- \sigma_{R_m}^*(m',r') - \sigma_K(-m',-r') \Big\}.
\end{equation}
Thus, in view of \eqref{eq:someEq} and \eqref{eq:someothereq}, there only remains to show that
\begin{equation}
\label{eq:ConjugateIdentity}
\sigma_{R_m}^*(m',r')= \iota_{R_m}(m',r'). 
\end{equation}
This identity will be proven by hand, based on the observation that $(\mathcal{C}_p^0(\Xcal) \times\Ccal_q^0(\Ycal_R))^*$ endowed with the weak-$^*$ topology is a locally convex Hausdorff topological vector space whose dual is exactly $\mathcal{C}_p^0(\Xcal) \times\Ccal_q^0(\Ycal_R)$, see for instance \cite[Chapter 3 -- Propositions 3.11 and 3.14]{Brezis}. 

First we notice that $\sigma_{R_m}(0,0)= 0$, so that
\begin{equation*}
\sigma_{R_m}^*(m',r') \geq \langle (m',r'),(0,0) \rangle - \sigma_{R_m}(0,0) = 0,
\end{equation*}
for every $(m',r')\in \Ccal_p^0(\Xcal)^* \times\Ccal_q^0(\Ycal_R)^*$. Assume next that $\sigma_{R_m}^*(m',r') > 0$ for some $(m',r')\in \Ccal_p^0(\Xcal)^* \times\Ccal_q^0(\Ycal_R)^*$, so there must exist a pair $(f,g) \in \Ccal_p^0(\Xcal) \times\Ccal_q^0(\Ycal_R)$ such that 
\begin{equation*}
\langle (m',r'),(f,g) \rangle - \sigma_{R_m}(f,g)>0.    
\end{equation*}
Upon noting that $\sigma_{R_m} \colon \Ccal_p^0(\Xcal) \times \Ccal_q^0(\Ycal_R) \to [0,+\infty]$ is positively homogeneous, this further implies that
\begin{align*}
\sigma_{R_m}^*(m',r')
& \geq \,\sup_{\alpha \geq 0} \big\langle (m',r'), (\alpha f, \alpha g) \big\rangle - \sigma_{R_m}(\alpha f, \alpha g) \\
& = \, \sup_{\alpha \geq 0} \alpha \Big( \big\langle (m',r'), (f,g) \big\rangle - \sigma_{R_m}(f,g) \Big) = +\infty, 
\end{align*}
whence the function $\sigma_{R_m}^*$ can only take the values $\{0,+\infty\}$. At this stage, take a pair $(m',r') \in R_m
$, and observe that by definition of the support function, one has that 
\begin{equation*}
\sigma_{R_m}(f,g) \geq \langle (m',r'), (f,g) \rangle    
\end{equation*}
for any $(f,g) \in \Ccal_p^0(\Xcal) \times\Ccal_q^0(\Ycal_R)$. By taking the supremum over all such couples, this implies that $\sigma_{R_m}^*(m',r') \leq 0$, so necessarily $\sigma_{R_m}^*(m',r')= 0$. Let us now pick $(m',r') \notin R_m$. Since $R \subset \Pcal_q(\Ycal_R)$ is convex and compact with respect to the weak-$^*$ topology of $\Ccal_p^0(\Ycal_R)^*$, the set $R_m =  \{m\} \times R$ is also convex and weakly-$^*$ compact. Recalling that the weak-$^*$ topology is locally convex, we may infer from Hahn-Banach's separation principle (see e.g.\@ \cite[Theorem 5.79]{Aliprantis2006}) the existence of an $\varepsilon > 0$ along with a couple $(f_{\epsilon},g_{\epsilon}) \in \Ccal_p^0(\Xcal) \times \Ccal_q^0(\Ycal_R)$ such that
\begin{equation*}
\langle (m',r'),(f_{\epsilon},g_{\epsilon}) \rangle
\geq \langle (m_0,r_0),(f_{\epsilon},g_{\epsilon}) \rangle + \varepsilon
\end{equation*}
for all $(m_0,r_0) \in R_m$. This directly implies that
\begin{equation*}
\langle (m',r'),(f_{\epsilon},g_{\epsilon}) \rangle
\geq \sigma_{R_m}(f_{\epsilon},g_{\epsilon}) + \varepsilon,
\end{equation*}
which in turn yields $\sigma_{R_m}^*(m',r') \geq \varepsilon > 0$, and so $\sigma_{R_m}^*(m',r')= +\infty$ for $(m',r') \notin R_m$. In summary, we have proven \eqref{eq:ConjugateIdentity}, which together with \eqref{eq:someEq} and \eqref{eq:someothereq} allows us to deduce that
\begin{equation*}
\begin{aligned}
\inf_{(f,g) \in K} \int_{\Xcal} f(x) \dd m(x) + \sigma_{R} (g)
& = \sup_{(m',r') \,\in\, \Ccal_p^0(\Xcal)^* \times \Ccal_q^0(\Ycal_R)^*} \Big\{ - \sigma_K(-m',-r') - \iota_{R_m}(m',r') \Big\} \\
& = \sup_{(m',r') \,\in\, R_m} - \sigma_K(-m',-r') \\
& = \chi_R(m),
\end{aligned}
\end{equation*}
and concludes the proof. 
\end{proof}

\begin{corollary}[Concavity of $\chi_R$]
Suppose that Assumptions \ref{hyp:OT}-\textnormal{(i)} and \text{(iii)} hold. Then, the map $\chi_R : \Pcal_p(\R) \to \R$ is concave.
\end{corollary}

\begin{proof}
The result could have been derived by means of gluing techniques similar to those utilized in Section \ref{section:ot-rm}. For the sake of concision, we simply notice that for given $g \in \mathcal{C}_q^0$, the cost function $\tilde{\Jcal}(g)$ is an affine function of $m \in\Pcal_p(\R)$. Thus by \eqref{eq:StrongDuality}, the function $\chi_R : \Pcal_p(\R) \to \R$ is the pointwise infimum of a family of affine functions, and therefore concave.
\end{proof}

We close this section by discussing with two concrete and highly relevant situations in which the dual problems at hand admit a solution.

\begin{theorem}[Existence of dual solutions for licorms with finite support]
\label{theo:existence_finitecase}
Suppose that Assumptions \ref{hyp:OT}-\textnormal{(i)} and \textnormal{(iii)} hold, and assume additionally that $\Ycal_R\subset\R_+$ is a finite set. Then, both problems \eqref{eq:first_dualpb} and \eqref{eq:second_dualpb} have a solution.
\end{theorem}

\begin{proof}
To begin with, note that by the definition of the dual problems \eqref{eq:first_dualpb} and \eqref{eq:second_dualpb}, it is sufficient to prove that the latter has a solution. Indeed, if $g \in\Ccal_q^0(\Ycal_R)$ is a solution of \eqref{eq:second_dualpb}, then $g^*$ is Lipschitz continuous as the Fenchel transform of a function with bounded domain. In particular $g^* \in\Ccal_p^0(\Xcal)$, and the pair $(g^*,g) \in K$ solves \eqref{eq:first_dualpb} since $\Jcal(g^*,g) = \tilde{\Jcal}(g)$.

Since $\Ycal_R \subset \R_+$ is assumed to be finite, we may describe it as $\Ycal_R \triangleq \{ y_0,\ldots,y_K \}$. To alleviate notations, we shall represent every $(g,r) \in \Ccal_q^0(\Ycal_R) \times \Pcal_q(\Ycal_R)$ by  
\begin{equation}
\label{eq:RepresentationEmpirical}
g \triangleq \sum_{k=0}^K g_k \mathds{1}_{y_k} \qquad \text{and} \qquad r \triangleq \sum_{k=0}^K r_k \delta_{y_k}
\end{equation}
for some tuples $(g_0,\ldots,g_K) \in\R^{K+1}$ and $(r_0,\ldots,r_K) \in\R^{K+1}$. Then, we begin by observing that for any $g \in \Ccal_q^0(\Ycal_R)$, one has that
\begin{equation*}
\int_{\Xcal} g^*(x) \dd m(x)
+ \sigma_R(g)=
\int_{\Xcal} \tilde{g}^*(x) \dd m(x)
+ \sigma_R(\tilde{g})
\end{equation*}
whenever $\tilde{g} \triangleq g + C$ for some constant $C \in\R$. Therefore, we assume that $g_0=0$ in the dual problem \eqref{eq:second_dualpb} without changing its value. Thus, the latter becomes equivalent to the finite-dimensional program
\begin{equation*}
\inf_{(g_1,\ldots,g_K) \in \R^K}
\hat{\Jcal}(g_1,\ldots,g_K) \triangleq \inf_{(g_1,\ldots,g_K) \in \R^K}
\int_{\Xcal} \, \left( \, \max_{k \in\{0,\ldots,K\}} \, xy_k- g_k \, \right) \dd m(x)
+ \sup_{r \in R}
\bigg( \sum_{k=1}^K g_k r_k \bigg).
\end{equation*}
Hence, to ensure the existence of a solution to the latter problem, it is sufficient to prove that the functional $\hat{\Jcal} \colon \R^K \to \R\cup\{+\infty\}$ is lower semicontinuous and coercive.

We begin by showing that $\hat{\Jcal} \colon \R^K \to \R\cup\{+\infty\}$ is Lipschitz continuous with respect to the supremum norm over $\R^K$. First, observe that the map
\begin{equation*}
(g_1,\ldots,g_K) \in \R^K \mapsto \sum_{k=1}^K g_k r_k \in \R    
\end{equation*}
is $1$-Lipschitz for any fixed tuple $(r_1,\ldots,r_K) \in\R^K$, so that the supremum with respect to $r \in R$ is again $1$-Lipschitz. By the same argument, the map 
\begin{equation*}
(g_1,\ldots,g_K) \in \R^K \mapsto \max_{k \in\{1,\ldots K\}} xy_k - g_k \, \in \R
\end{equation*}
is also $1$-Lipschitz for any $x \in \Xcal$, from whence we easily deduce that $\hat{\Jcal} \colon \R^K \to \R \cup \{+\infty\}$ is Lipschitz continuous. We next show that $\hat{\Jcal} \colon \R^K \to \R \cup \{+\infty\}$ is coercive. Recalling our convention \eqref{eq:RepresentationEmpirical} for representing empirical measures, it is easy to see that for any $k \in \{0,\ldots,K\}$, there exists some $r^{(k)} \in R$ such that $r_k^{(k)}> 0$. Next, we consider
\begin{equation*}
\bar{r} \triangleq \frac{1}{K+1} \sum_{k=0}^K r^{(k)} \in R
\end{equation*}
where we used the fact that $R \subset\Pcal_q(\Ycal_R)$ is convex, and note that $\bar{r}_k > 0$ for all $k\in\{0,\ldots,K\}$ by construction. For any $g \in \R^{K+1}$, it further holds that
\begin{equation} \label{eq:boundgk1}
\sigma_R(g) \geq \sum_{k=0}^K g_k \bar{r}_k.
\end{equation}
Then, letting $\varepsilon \triangleq \bar{r}_0/K>0$, we define
\begin{equation*}
s_k = \left\{
\begin{aligned}
& \bar{r}_k + \varepsilon &  \text{if $g_k \leq 0$}, \\
& 0 & \text{otherwise},
\end{aligned}
\right.
\end{equation*}
for each $k\in\{1,\ldots,K\}$, and remark that 
\begin{equation*}
\sum_{k=1}^K s_k \leq K\varepsilon + \sum_{k=1}^K \bar{r}_k = \sum_{k=0}^K \bar{r}_k = 1.    
\end{equation*}
At this stage, set $s_0\triangleq 1- \sum_{k=1}^K s_k \geq 0$, and note that 
\begin{align}
\int_{\Xcal}
\, \left( \, \max_{k \in\{0,\ldots,K\}}  xy_k- g_k \right) \dd m(x)
\geq {} & \int_{\Xcal}
\sum_{k=0}^K s_k (xy_k- g_k) \dd m(x) \notag \\
= {} &
\Big( \sum_{k=0}^K s_k y_k \Big) \int_{\Xcal} x \dd m(x)
- \sum_{k=1}^K s_k g_k
\notag \\
\geq {} &  - C - \sum_{k=1}^K s_k g_k, \label{eq:boundgk2}
\end{align}
where $C \triangleq \big( \max_{k\in\{0,\ldots K\}} |y_k| \big) |\E[m]|$. Combining the lower bounds \eqref{eq:boundgk1} and \eqref{eq:boundgk2} while using the definition of the tuple $(s_0,\ldots,s_K) \in \R^{K+1}$, we finally obtain that
\begin{equation*}
\hat{\Jcal}(g_1,\dots,g_K) \geq - C + \sum_{k=1}^K (\bar{r}_k - s_k) g_k \geq - C + \sum_{k=1}^K \min\{\bar{r}_k,\varepsilon\} |g_k|,
\end{equation*}
which proves the coercivity of $\hat{\Jcal} \colon \R^K \to \R \cup\{+\infty\}$ and concludes the proof.
\end{proof}

We now deal with the case in which $m \in\Pcal_p(\R)$ has a bounded support and $\mathcal{Y}_R \subset \R$ is bounded. The proof is inspired by that of \cite[Proposition 9.16]{bonnans2019convex}, and follows standard arguments subtending the existence of Kantorovich potentials for classical optimal transport problems.

\begin{theorem}[Existence of dual solutions for licorms with bounded support] \label{theo:existence_bounded}
Suppose that Assumptions \ref{hyp:OT}-\textnormal{(i)} and \textnormal{(iii)} hold. In addition, assume that $m\in\Pcal_p(\R)$ has a bounded support and that the set $\mathcal{Y}_R \subset\R$ is bounded. Then, both dual problems \eqref{eq:first_dualpb} and \eqref{eq:second_dualpb} have a solution.
\end{theorem}

\begin{proof}
It is sufficient to prove to the existence of a solution $(f,g) \in K$ to  \eqref{eq:first_dualpb}, since then $g \in \Ccal_q^0(\Ycal_R)$ is also a solution to \eqref{eq:second_dualpb}. Denoting by $L_{\Xcal} \triangleq \sup_{x \in \Xcal} |x|$ and $L_{\Ycal_R} \triangleq \sup_{y \in \Ycal_R} |y|$ where $\Xcal\triangleq\supp(m)$, we take a minimizing sequence $((f_k,g_k))_{k \in \mathbb{N}} \subset K$ and define
\begin{equation*}
\tilde{f}_k(x) \triangleq g_k^*(x) \qquad \text{and} \qquad \tilde{g}_k(y) \triangleq \tilde{f}_k^*(y),
\end{equation*}
for every $x \in \mathcal{X}$ and $y \in \mathcal{Y}_R$ and for each $k \in\N$. We claim that $((\tilde{f}_k,\tilde{g}_k))_{k \in \mathbb{N}}$ is also a minimizing sequence for \eqref{eq:first_dualpb}. First, note that since 
\begin{equation*}
\tilde{f}_k(x) = \sup_{y \in \Ycal_R} xy - g_k(y),
\end{equation*}
for all $k \in \mathbb{N}$, the latter is $L_{\Ycal_R}$-Lipschitz as the pointwise supremum of a family of $L_{\Ycal_R}$-Lipschitz functions, and in particular $\tilde{f_k} \in \Ccal_p^0(\Xcal)$. For the same reason, one also has that $\tilde{g}_k \in \Ccal_q^0(\Ycal_R)$ is $L_{\Ycal_R}$-Lipschitz for each $k \in\N$, and that $(\tilde{f}_k,g_k) \in K$ as a direct consequence of the Fenchel-Young inequality. It then stems from what precedes along with basic facts on Fenchel conjugates (see e.g.\@ \cite[Section 6.1]{Hiriart1996}) that  $(\tilde{f}_k,\tilde{g}_k) \in K$ for each $k \in \mathbb{N}$, with 
\begin{equation*}
\tilde{f}_k \leq f_k \qquad \text{and} \qquad \tilde{g}_k = (g_k^*)^* \leq g_k.
\end{equation*}
It is then easy to see that
\begin{equation*}
\Jcal(\tilde{f}_k,\tilde{g}_k) \leq \Jcal(\tilde{f}_k,g_k) \leq  \Jcal(f_k,g_k),  
\end{equation*}
which implies that $((\tilde{f}_k,\tilde{g}_k))_{k \in \mathbb{N}} \subset K$ is a minimizing sequence for problem \eqref{eq:first_dualpb}.

We have already proven that $(\tilde{f}_k)_{k\in\N} \subset \Ccal_p^0(\Xcal)$ and $(\tilde{g}_k)_{k\in\N} \subset\Ccal_q^0(\Ycal_R)$ are sequences of equi-Lipschitz functions, with constant $L_{\Xcal}$ and $L_{\Ycal_R}$ respectively. Moreover, it can be shown by elementary computations that $(\tilde{f}_k-C)^* = \tilde{g}_k +C$ and 
\begin{equation*}
\Jcal(\tilde{f}_k - C, \tilde{g}_k + C) = \Jcal(\tilde{f}_k, \tilde{g}_k)
\end{equation*}
for every $C \in\R$. Fixing some $x_0 \in \Xcal$ and taking $C \triangleq \tilde{f}_k(x_0)$ in the previous identity, we may thus posit that $\tilde{f}_k(x_0) = 0$. This, together with the preceding uniform Lipschitz bound and the fact that $\Xcal\subset \R$ is compact, entails that
\begin{equation*}
\sup_{k \in \N} \max_{x \in \Xcal} |\tilde{f}_k(x)| < +\infty.
\end{equation*}
Recalling that $\tilde{g}_k = \tilde{f}_k^*$ and that $\Ycal_R \subset \R$ is also a compact set, this further implies 
\begin{equation*}
\sup_{k \in \N} \max_{y \in \Ycal_R} |\tilde{g}_k(y)| < +\infty.
\end{equation*}
Therefore, we can apply the Ascoli-Arzelà theorem (see e.g.\@ \cite[Theorem 11.28]{Rudin1987}) to infer the existence of a subsequence of $((\tilde{f}_k,\tilde{g}_k))_{k \in \mathbb{N}} \subset K$ that we do not relabel, which converges uniformly towards some $(\bar{f},\bar{g}) \in \Ccal_p^0(\Xcal) \times \Ccal_q^0(\Ycal_R)$. Observing that the set $K \subset \Ccal_p^0(\Xcal) \times\Ccal_q^0(\Ycal_R)$ is closed for this topology, we further have that $(\bar{f},\bar{g}) \in K$. Finally, as it can be shown quite easily that $\Jcal \colon \Ccal_p^0(\Xcal) \times \Ccal_q^0(\Ycal_R) \to \R$ is Lipschitz continuous in the supremum norm, we obtain 
\begin{equation*}
\Jcal(\bar{f},\bar{g}) = \lim_{k \to +\infty} \Jcal(\tilde{f}_k,\tilde{g}_k)
\leq \lim_{k \to +\infty} \Jcal(f_k,g_k)
= \inf_{(f,g) \in K} \Jcal(f,g)
\end{equation*}
which thereby shows the optimality of the pair $(\bar{f},\bar{g}) \in K$ and concludes the proof.
\end{proof}


\section{Examples}
\label{section:examples}

In this section, we show how a large class of licorms which are highly relevant in applications fit in the framework we developed.  

\subsection{Conditional Value at Risk}

In this subsection, we begin by we analyzing the famed Conditional Value at Risk, which was briefly discussed in Remark \ref{rmk:Intro} of the introduction. In the sequel, we fix a probability level $\beta \in [0,1)$, and the definition of the mapping $\mathrm{CV@R}_{\beta} \colon \mathbb{L}^1(\Omega,\R) \rightarrow \R$, given by
\begin{equation*}
\mathrm{CV@R}_{\beta}(X) \triangleq \inf_{t \in \R} \left\{ t + \frac{1}{1-\beta} \mathbb{E} \big[ (X-t)_+ \big] \right\},
\end{equation*}
for every $X \in \mathbb{L}^1(\Omega,\R)$.
Next we define $R_{\beta} \triangleq \{ r_\beta \}$ where
\begin{equation*}
r_\beta \triangleq \beta \delta_0
+ (1-\beta) \delta_{1/(1-\beta)}.
\end{equation*}
It is quite obvious that $R_{\beta} \subset\mathcal{P}_c(\R)$ satisfies Assumptions \ref{hyp:OT}-(i), (ii) and (iii), but also that the associated support $\Ycal_{R_{\beta}} =\{ 0, 1/(1-\beta) \} \subset \R_+$ is a finite set. Below, we show that the latter set allows for an optimal transport representation of the conditional value at risk. 

\begin{proposition}[Optimal transport representation of $\mathrm{CV@R_{\beta}}$]
For any $X \in \mathbb{L}^1(\Omega,\R)$, it holds that
\begin{equation*}
\mathrm{CV@R}_{\beta}(X)= \rho_{R_\beta}(X).
\end{equation*}
\end{proposition}

\begin{proof}
By the duality result of Theorem \ref{theo:duality}, we know that 
\begin{equation*}
\rho_{R_\beta}(X)
= \inf_{g \in \Ccal_b(\{ 0, 1/(1-\beta)\})} \mathbb{E}\big[ g^*(X) \big]
+ \sigma_{R_\beta}(g(X)).
\end{equation*}
Furthermore, as discussed in the proof of Theorem \ref{theo:existence_finitecase}, we can require that $g(0)=0$ without modifying the value of the problem, which then boils down to optimizing over the parameter $\theta \triangleq g(1/(1-\beta))$. More precisely, we have that
\begin{equation*}
\begin{aligned}
\rho_{R_\beta}(X) & = \inf_{\theta \in \R} \Big\{ \E \bigg[ \sup_{y \in\R} X y - g(y) \Big] + \INTDom{g(x)}{\R}{r_{\beta}(x)} \bigg\} \\
& = \inf_{\theta \in \R}
\left\{
\mathbb{E} \Big[ \max \Big(0, \frac{X}{1-\beta} - \theta \Big) \Big] + (1-\beta) \theta \right\},
\end{aligned}
\end{equation*}
at which point the change of variable $ t\triangleq (1-\beta) \theta$ yields the desired identity.
\end{proof}

\subsection{Higher moment measures}

In this subsection we fix some $p \in (1,+\infty)$ along with a constant $c > 1$, and focus on the licorm $\rho_{p,c} \colon \mathbb{L}^p(\Omega,\R) \rightarrow \R$ defined by
\begin{equation}
\label{eq:def_high_order_measure}
\rho_{p,c}(X)
\triangleq \inf_{t \in \R} \, \left\{ t + c \, \mathbb{E} \big[ (X-t)_+^p \big]^{1/p} \right\}.
\end{equation}
This risk measure is referred to as \emph{higher order dual risk measure} in \cite{dentcheva2010kusuoka}, and is known to be a licorm (see e.g.\@ \cite[Example 1.4 and Theorem 1]{krokhmal2007higher}). The main result of this subsection is an optimal transport representation of $\rho_{p,c}$ in terms of the subset $R_{q,c} \subset \mathcal{P}_q(\R_+)$ defined by
\begin{equation*}
R_{q,c}
\triangleq \left\{ r \in \mathcal{P}(\R_+)  ~\,\textnormal{s.t.}~
\int_{\R_+} y \dd r(y)= 1 ~~\text{and}~~
\int_{\R_+} y^q \dd r(y) \leq c^q
\right\}.
\end{equation*}
It can be easily checked that this set satisfies Assumptions \ref{hyp:OT}-(i), (ii), and (iii).

\begin{proposition}[Optimal transport representation of higher moment risk measures]
\label{prop:ot_rep_higher}
For all $p \in (1,+\infty)$ and $c>1$, it holds that
\begin{equation}
\label{eq:formula_high}
\rho_{p,c}(X)
= \rho_{R_{q,c}}(X)
\end{equation}
for each $X \in \Lbb^p(\Omega,\R)$.
\end{proposition}

\begin{proof}
The proof of this result will rely on computing the set $\mathfrak{Y} \subset\Lbb^q(\Omega,\R)$ involved in the standard dual representation \eqref{eq:dual_repbis}  of the risk measure $\rho_{p,c} : \Lbb^p(\Omega,\R) \to \R$, namely
\begin{equation*}
\rho_{p,c}(X) = \sup_{Y \in\mathfrak{Y}} \E[XY]. 
\end{equation*}
We shall see that the latter is exactly the domain of the Fenchel conjugate $\rho_{p,c}^*$. To begin with, let $Y \in \mathbb{L}^q(\Omega,\R)$ be a nonnegative random variable of expectation equal to 1, and note that 
\begin{align*}
\rho_{p,c}^*(Y)
= \sup_{\begin{subarray}{c} X \in \mathbb{L}^p(\Omega,\R) \\ t \in \R \end{subarray}}
\Big\{
\mathbb{E}[ XY ] -t - c \mathbb{E} \big[ (X-t)_+^p \big]^{1/p} \Big\}
= \sup_{Z \in \mathbb{L}^p(\Omega,\R)}
\Big\{
\mathbb{E}[ ZY ] - c \mathbb{E} \big[ Z_+^p \big]^{1/p} \Big\}
\end{align*}
where the second equality follows from the change of variable $Z \triangleq X+t$. It is then easy to realize that, since $Y \geq 0$ almost surely, we have
\begin{equation*}
\rho_{p,c}^*(Y)
= \sup_{Z \in \mathbb{L}^p(\Omega,\R)}
\Big\{
\mathbb{E}[ ZY ] - c \| Z \|_p \Big\}
= \sup_{A \geq 0} \,
\sup_{\begin{subarray}{c} X \in \mathbb{L}^p(\Omega,\R) \\ \| X \|_p = A \end{subarray}}
\mathbb{E}[ZY] - c A.
\end{equation*}
Combined the basic duality formula $\sup_{\| X \|_p = A}
\mathbb{E}\big[XY\big] = A \| Y \|_q$, we further obtain
\begin{equation*}
\rho_{p,c}^*(Y)
= \sup_{A \geq 0} A \big( \| Y \|_q - c \big)
= \left\{
\begin{aligned}
& 0 & \text{if $\| Y \|_q \leq c$}, \\
& + \infty & \text{otherwise,}
\end{aligned}
\right.
\end{equation*}
which combined with the standard dual representation of licorms recollected above entails that
\begin{equation*}
\mathfrak{Y}
= \Big\{
Y \in \mathbb{L}^q(\Omega,\R)
~\,\text{s.t.}~ Y \geq 0,~ \mathbb{E}\big[ Y \big]=1 ~\text{and}~  \| Y \|_q \leq c
\Big\}.
\end{equation*}
Upon observing that $\big\{ \mathbb{P}_Y ~\,\text{s.t.}~ Y \in \mathfrak{Y} \big\} \subseteq R_{q,c}$, we immediately infer from the inequality \eqref{eq:kusuoka_plus1} in Proposition \ref{prop:kusuoka_refined} that $\rho_{p,c} \leq \rho_{R_{q,c}}$. To prove the converse inequality, we take some $X \in \mathbb{L}^p(\Omega,\R)$, fix any $\pi \in \Pi(\P_X,R_{q,c})$ and choose some $t \in \R$. Then, we have that
\begin{align*}
\int_{\R^2} xy \dd \pi(x,y)
= {} & \int_{\R^2} (x-t)y \dd \pi(x,y) + t \\
\leq {} & \int_{\R^2} (x-t)_+ y \dd \pi(x,y) + t \\
\leq {} & \Big( \int_{\R^2} (x-t)_+^p \dd \pi(x,y)  \Big)^{1/p} \Big( \int_{\R^2} y^q \dd \pi(x,y) \Big)^{1/q} + t \\
\leq {} & c\, \mathbb{E}\big[ (X-t)_+^p \big]^{1/p} + t,
\end{align*}
where first equality derives from the fact that the $\E[\efrak^2_{\sharp}\pi]=1$, and the first inequality is a consequence of the fact that $\text{supp}(\pi) \subset \R \times \R_+$. The second inequality stems simply from H\"older's inequality, while the last inequality finally follows since $\pi \in \Pi(\P_X,R_{q,c})$. Maximizing the left-hand side with respect to $\pi \in \Pi(m,R_{q,c})$ and minimizing the right-hand side with respect to $t \in\R$ then yields the desired inequality. 
\end{proof}

The set $R_{q,c} \subset\Pcal_q(\R)$ introduced above being convex and closed for the weak-$^*$ topology, we may apply the duality results of Section \ref{section:duality} to the licorm $\rho_{p,c} : \Lbb^p(\Omega,\R) \to \R$. In the next proposition, we show the existence of dual solutions, and show that they take a rather particular form. To do so , given $(t,u) \in \R \times \R_+$, we consider the function $g_{t,u} : \R_+\to\R$ defined as
\begin{equation}
\label{eq:gtuDef}
g_{t,u}(y) \triangleq ty + uy^q,
\end{equation}
for every $y \in \R_+$.

\begin{proposition}[Duality formulas for higher moment risk measures]
\label{prop:higher_moment_duality}
For any $X \in \mathbb{L}^p(\Omega,\R)$, it holds that
\begin{equation}
\label{eq:prop:higher}
\rho_{p,c}(X)
= \inf_{g \in \mathcal{C}_q^0(\mathcal{Y}_R)} \mathbb{E}\big[ g^*(X) \big] + \sigma_R(g)
= \inf_{(t,u) \in \R \times\R_+} \mathbb{E}\big[ g_{t,u}^*(X) \big] + t + uc^q.
\end{equation}
In addition, there exists a minimizer $(\bar{t},\bar{u}) \in \R \times \R_+$ for the second optimization problem in \eqref{eq:prop:higher} such that $g_{\bar{t},\bar{u}}$ is a minimizer of the first one.
\end{proposition}

The proof of this proposition relies on the following two technical lemmas, in which we explicitly compute the support function $\sigma_{R_{q,c}} \colon \Ccal^0_p(\R_+) \to [0,+\infty]$ of the set $R_{q,c}\subset\Pcal_q(\R)$ introduced above, and the Fenchel conjugate of $g_{t,u} : \R_+ \to \R$.

\begin{lemma}[Dual formulation of the support function]
\label{lemma:computation_sigma}
For all $g \in \mathcal{C}_p^0(\R_+)$, it holds that
\begin{equation}
\label{eq:formula_sigmaRq}
\sigma_{R_{q,c}}(g)
= 
\left\{
\begin{aligned}
\inf_{(s,t,u) \in \R^3} & \, s + t + u c^q, \\
\mathrm{s.t.} \,\,~~ & \left\{
\begin{aligned}
& g(y) \leq s + ty + uy^q ~~ \text{for all $y \in \R_+$}, \\
& u \geq 0.
\end{aligned}
\right.
\end{aligned}
\right.
\end{equation}
\end{lemma}

\begin{proof}
The proof of this result essentially consists in applying the Fenchel-Rockafellar duality theorem (see e.g. \cite{Rockafellar1966}) to the minimization problem in the right-hand side of \eqref{eq:formula_sigmaRq}, where $g \in\Ccal^0_p(\R_+)$ is given. To this end, we first need to fix a few notations. Consider first the functions $e_{\ell} \colon y \in\R_+ \mapsto y^{\ell} \in \R_+$ defined for each $\ell \in \R_+$, as well as the bounded linear operator $A \colon \R^3 \rightarrow \mathcal{C}_q^0(\R_+)$ given by
\begin{equation*}
A(s,t,u)= se_0 + te_1 + ue_q.
\end{equation*}
Next, denote by $\mathcal{C}_q^0(\R_+,\R_+)$ the subset of nonnegative functions in $\mathcal{C}_q^0(\R_+)$, and consider the mapping $G \colon \mathcal{C}_q^0(\R_+) \rightarrow [0,+\infty]$ given by 
\begin{equation*}
G(h) \triangleq \iota_{\mathcal{C}_q^0(\R_+,\R_+)}(h-g),
\end{equation*}
for every $h \in \mathcal{C}_q^0(\R_+)$. Finally, define the extended real-valued function $F \colon \R^3 \rightarrow \R \cup \{ + \infty \}$ as 
\begin{equation*}
F(s,t,u) \triangleq s+ t + u c^q + \iota_{\R_+}(u)
\end{equation*}
for every $(s,t,u) \in \R^3$. With these notations, the minimization problem in \eqref{eq:formula_sigmaRq} can writes as
\begin{equation*}
\inf_{(s,t,u) \in \R^3} F(s,t,u) + G(A(s,t,u)).
\end{equation*}
Since the functions $F$ and $G$ are clearly proper, lower semicontinuous and convex functions, we may apply Fenchel-Rockafellar's duality theorem \cite{Rockafellar1966} provided that $0 \in \text{int}(\Xi)$, where
\begin{align*}
\Xi \triangleq {} &
A \, \text{dom}(F)
- \text{dom}(G) \\
= {} & \bigg\{
\tilde{h} \in \mathcal{C}_q^0(\R_+)
~\,\text{s.t.}~
\tilde{h} = se_0 + te_1 + u e_q - h - g ~~\text{for some}~
(t,s,u) \in\R^2 \times\R_+ \\
& \hspace{9.7cm} \text{and}~ h \in \mathcal{C}_q^0(\R_+, \R_+) 
\bigg\}.
\end{align*}
Take any $\tilde{h} \in \mathcal{C}_q^0(\R_+,\R_+)$, and note that necessarily $g + \tilde{h} \in\mathcal{C}_q^0(\R_+)$. By definition, this implies in particular that there exists some $u \geq 0$ such that 
\begin{equation*}
g(y) + \tilde{h}(y) \leq u e_0(y) + u e_q(y)    
\end{equation*}
for every $y \in \R_+$. Define then $h \triangleq (ue_0 + u e_q) - (g+ \tilde{h})$, and note that $h \in \mathcal{C}_q^0(\R_+,\R_+)$ quite clearly from what precedes. In summary, we have shown that $\tilde{h} \in \Xi$ with $(s,t,u) = (u,0,u)$ and $h \in C^0_q(\R_+,\R_+)$ given above, and since the latter was arbitrary it immediately follows that $\Xi = \mathcal{C}_q^0(\R_+)$, so that in particular $0 \in \text{int}(\Xi)$. We may thus apply Fenchel-Rockafellar's duality theorem to obtain
\begin{equation}
\label{eq:FenchelRockafellar1}
\inf_{(s,t,u) \in \R^3} F(s,t,u) + G(A(s,t,u))
= \sup_{r \in \mathcal{C}_q^0(\R_+)^*}
- F^*(A^*r) - G^*(-r)
\end{equation}
We are now left with computing the Fenchel conjugates appearing in the previous expression. This is the matter of elementary computations, through which one can show that 
\begin{equation}
\label{eq:FenchelRockafellar2}
\left\{
\begin{aligned}
& F^*(\alpha,\beta,\gamma)= \iota_{\{ 1 \} \times \{ 1 \} \times (-\infty, c^q]}(\alpha,\beta,\gamma), \\
& G^*(r)= \int_{\R_+} g(y) \dd r(y) + \iota_{-\mathcal{M}_+(\R_+) \cap \mathcal{C}_q^0(\R_+)^*}(r), \\
& A^* r = \left( \textstyle{\int_{\R_+}} e_0(y) \dd r(y), \textstyle{\int_{\R_+}} e_1(y) \dd r(y), \textstyle{\int_{\R_+}} e_q(y) \dd r(y) \right).
\end{aligned}
\right.
\end{equation}
for every $(\alpha,\beta,\gamma) \in\R^3$ and each $r \in \Ccal^0_q(\R_+)^*$. In the expression of $G^*(r)$, the symbol $\mathcal{M}_+(\R_+)$ refers to the set of nonnegative Radon measures over $\R_+$. It then follows from \eqref{eq:FenchelRockafellar2} that
\begin{align*}
\sup_{r \in \mathcal{C}_q^0(\R_+)^*}
- F^*(A^*r) - G^*(-r)
& = \left\{
\begin{aligned}
& \sup_{r \in \mathcal{M}_+(\R_+) \cap \mathcal{C}_q^0(\R_+)^*}
\int_{\R_+} g(y) \dd r(y), \\
& \hspace{1.25cm}  \mathrm{s.t.} \hspace{0.8cm} 
\begin{cases}
\int_{\R_+} e_0(y) \dd r(y) = \int_{\R_+} e_1(y) \dd r(y) = 1 \\
\int_{\R_+} e_q(y) \dd r(y) \leq c^q
\end{cases}
\end{aligned}
\right.
\\
& =  \sigma_{R_{c,q}}(g),
\end{align*}
which combined with\eqref{eq:FenchelRockafellar1} finally yields
\begin{equation*}
\inf_{(s,t,u) \in \R^3} F(s,t,u) + G(A(s,t,u)) = \sigma_{R_{c,q}}(g)  
\end{equation*}
as was to be shown.
\end{proof}

\begin{lemma}[Fenchel conjugate of $g_{t,u}$]
\label{lemma:fenchel_power_bis}
For any $(t,u) \in \R \times\R_+$, it holds that
\begin{equation*}
g_{t,u}^*(x) = \frac{1}{p} (uq)^{-(p-1)} (x-t)_+^p
\end{equation*}
for every $x \in \R$.
\end{lemma}

\begin{proof}
We begin by the situation in which $t=0$. Recalling the definition \eqref{eq:gtuDef} of the map $g_{t,u} : \R_+\to\R$, it holds in this case that 
\begin{equation*}
g_{0,u}(y)= (uq) g_{0,1/q}(y)
\end{equation*}
for every $y\in\R_+$, from which it is easily deduced by using classical homogeneity results on Fenchel conjugates that
\begin{equation*}
g_{0,u}^*(x)= (uq) g^*_{0,1/q} \Big(\frac{x}{uq} \Big).     
\end{equation*}
As it can be straightforwardly verified that 
\begin{equation*}
g_{0,1/q}^*(x)
= \sup_{y \geq 0} \, \Big\{ xy - \frac{1}{q} y^q \Big\} = \frac{1}{p} x_+^p,
\end{equation*}
it follows from the above considerations that
\begin{equation*}
g_{0,u}^*(x)= \frac{uq}{p} \left( \frac{x_+}{uq} \right)^p
= \frac{1}{p} (uq)^{-(p-1)} (x_+)^p 
\end{equation*}
which settles the case $t=0$. For arbitrary $t \in\R$, it is enough to notice that $g_{t,u}^*(x)= g_{0,u}^*(x-t)$, and the desired identity immediately follows.
\end{proof}

\begin{proof}[Proof of Proposition \ref{prop:higher_moment_duality}]
\label{proof:higher_moment}
The first equality in \eqref{eq:prop:higher} directly follows from Theorem \ref{theo:duality}. Regarding the second one, we note first that due to Lemma \ref{lemma:computation_sigma}, it holds that
\begin{equation*}
\rho(X)
= 
\left\{
\begin{aligned}
\inf_{
\begin{subarray}{c}
g \in \mathcal{C}_q^0(\R_+) \\[0.1em]
(s,t,u) \in \R^3
\end{subarray}} & \, \E\big[g^*(X) \big] + s + t + u c^q, \\
\mathrm{s.t.} \,\,~~ & \left\{
\begin{aligned}
& g(y) \leq s + ty + uy^q ~~ \text{for all $y \in \R_+$}, \\
& u \geq 0.
\end{aligned}
\right.
\end{aligned}
\right.
\end{equation*}
Given any feasible tuple $(g,s,t,u) \in\Ccal^0_q(\R_+) \times\R^2 \times\R_+$ for the above problem, note at first that $(\tilde{g},\tilde{s},\tilde{t},\tilde{u}) \triangleq(g-s,0,t,u)$ is also feasible with the same cost. Consequently, we may set $s = 0$ in the minimization problem. Observe next that the shifted map $\tilde{g} \in\Ccal^0_q(\R_+)$ can be replaced by the map $g_{t,u} \in \mathcal{C}_q^0(\R_+)$ defined in \eqref{eq:gtuDef}. Indeed, it clearly holds that $\tilde{g} \leq g_{t,u}$, and so $g_{t,u}^* \leq \tilde{g}^*$, which then achieves a smaller cost. In summary, we have shown that
\begin{equation}
\label{eq:intermediate_pb}
\rho_{R_{q,c}}(X)
= \inf_{(t,u) \in\R\times\R_+}
\mathbb{E} \big[ g_{t,u}^*(X) \big] + t + uc^q.
\end{equation}
for every $X \in\Lbb^p(\Omega,\R)$, which is indeed the second equality in \eqref{eq:prop:higher}. 

To conclude, there remains to show the existence of solutions to both problems in \eqref{eq:prop:higher}. To this end, recalling the definition \eqref{eq:def_high_order_measure} of $\rho_{p,c} : \Lbb^p(\Omega,\R) \to \R$, it follows e.g. from \cite[Example 1.4 and Theorem 1]{krokhmal2007higher} that for each $X \in\Lbb^p(\Omega,\R)$, there exists some $\bar{t} \in\R$ such that
\begin{equation*}
\rho_{p,c}(X) = c \, \mathbb{E} \big[ (X-\bar{t})_+^p \big]^{1/p} + \bar{t}.
\end{equation*}
Let us then set $\bar{a} \triangleq \mathbb{E}[ (X-\bar{t})_+^p]$ and consider the competitor
\begin{equation*}
\bar{u} \triangleq \frac{1}{q c^{q-1}} a^{1/p}
= \frac{1}{q c^{q-1}} \mathbb{E}\big[ (X-\bar{t})_+^p \big]^{1/p}.
\end{equation*}
We first prove the existence of minimizers in the case where $\bar{a}=0$. This condition means that $X \leq \bar{t}$ almost surely and so $\rho_{p,c}(X)=\bar{t}$. It then follows from its very definition that $\bar{u}=0$, so that $g_{\bar{t},\bar{u}}^*= \iota_{(-\infty,\bar{t}]}$ and 
\begin{equation*}
\begin{aligned}
\mathbb{E} \big[ g_{\bar{t},\bar{u}}^*(X) \big] + \bar{t} + \bar{u}c^q & = \bar{t} \\
& = \rho_{p,c}(X).
\end{aligned}
\end{equation*}
where we again used that $X \leq \bar{t}$ almost surely. This proves the optimality of $(\bar{t},\bar{u})$ for the second problem in \eqref{eq:prop:higher}. The fact that $g_{\bar{t},\bar{u}}$ is a minimizer for the first one follows then from Lemma \ref{lemma:computation_sigma}. If we assume now that $\bar{a}>0$, it is easy to verify from Lemma \ref{lemma:fenchel_power_bis} above that
\begin{equation*}
g_{\bar{t},\bar{u}}^*(x)
= \frac{c}{p} \bar{a}^{\frac{1}{p}-1}(x-\bar{t})_+^p,
\end{equation*}
for every $x \in \R$.
Upon noting that $uc^q= \frac{c}{q}a^{1/p}$ as well, it follows that 
\begin{equation*}
\begin{aligned}
\mathbb{E} \big[ g_{\bar{t},\bar{u}}^*(X) \big] + \bar{t} + \bar{u}c^q
& = \frac{c}{p}\bar{a}^{\frac{1}{p}-1}\bar{a}
+ \frac{c}{q}\bar{a}^{1/p} + \bar{t} \\
& = c\bar{a}^{1/p} + \bar{t} \\
& = \rho_{p,c}(X).
\end{aligned}
\end{equation*}
This proves the optimality of $(\bar{t},\bar{u}) \in\R\times\R_+$ for the second problem in \eqref{eq:prop:higher}. The fact that $g_{\bar{t},\bar{u}} \in\Ccal^0_q(\R_+)$ is a minimizer for the first one follows again from Lemma \ref{lemma:computation_sigma}.
\end{proof}

\begin{remark}[Generalization to $\phi$-divergences]
We conjecture that the arguments developed above to study higher moment measures can be generalized to arbitrary \textit{$\phi$-divergences} (see \cite{ben1987penalty,shapiro2017distributionally}). Given a convex and lower semicontinuous function $\phi \colon \R_+ \rightarrow \R \cup \{ + \infty \}$ with $\phi(1)=0$, the term $\phi$-divergence commonly refers to the licorm $\rho : \Lbb^p(\Omega,\R) \to\R \cup\{+\infty\}$ defined by
\begin{equation*}
\rho(X)
\triangleq
\sup_{Y \in\Lbb^q(\Omega,\R)} \bigg\{ \mathbb{E}[XY ] ~\,\textnormal{s.t.}~ Y \geq 0, ~~ \mathbb{E}[ Y]= 1 ~~\text{and}~~ \mathbb{E}[\phi(Y)] \leq c \bigg\},
\end{equation*}
for some $c> 0$. Note that the latter can be made real-valued as soon as $\phi : \R_+ \to \R \cup\{+\infty\}$ satisfies some suitable growth conditions. Moreover, it was shown in \cite[Theorem 5.1]{ahmadi2012entropic} that
\begin{equation}
\label{eq:duality_phi}
\rho(X)= \inf_{(t,u) \in \R \times\R_+}
\mathbb{E}\big[ (\lambda \phi)^*(X-t) \big] + t + uc.
\end{equation}
Given the above, it is therefore reasonable to infer that $\rho= \rho_R$ with
\begin{equation*}
R= \bigg\{ r \in \mathcal{P}(\R_+) ~\,\text{s.t.}~
\mathbb{E}[r]= 1 ~\text{and}~ \int_{\R_+}\phi(y) \dd r(y) \leq c \bigg\}.
\end{equation*}
Indeed, informal computations (that would need to be made precise in the spirit of Lemma \ref{lemma:computation_sigma} above) suggest that
\begin{equation*}
\sigma_{R_{q,c}}(g)
= 
\left\{
\begin{aligned}
\inf_{(s,t,u) \in \R^3} & \, s + t + u c^q, \\
\mathrm{s.t.} \,\,~~ & \left\{
\begin{aligned}
& g(y) \leq s + ty + u \phi(y) ~~ \text{for all $y \in \R_+$}, \\
& u \geq 0.
\end{aligned}
\right.
\end{aligned}
\right.
\end{equation*}
which would then allow to recover the duality formula \eqref{eq:duality_phi}, thanks to Theorem \ref{theo:duality}.
\end{remark}

\section{Conclusion}

We have established a new characterization of law-invariant and coherent risk measures through a generalized optimal transport problem between the probability distribution of the random variable of interest, and a given target set $R \subset\Pcal_q(\R)$. Our representation is strongly related to the famed Kusuoka theorem, and is of particular interest when the set $R$ is convex. In this case, the associated risk measure $\rho_R : \Lbb^p(\Omega,\R) \to\R$ can be expressed in terms of a minimization problem, which can be utilized in a numerical perspective. For example, assuming that $R$ is convex, the Moreau envelope of $\rho_R$ can also be represented as a minimization problem, which can then be leveraged e.g. in the augmented-Lagrangian approach developed in \cite{kouri2022primal}. Future work may also focus on the extension of the progressive hedging algorithm introduced in \cite{rockafellar2002conditional}, which was adapted to the Conditional Value-at-Risk in \cite{rockafellar2018solving}. 

\bibliographystyle{plain}
{\footnotesize 
\bibliography{ref}}

\end{document}